\documentclass[preprint,12pt]{elsarticle}

\usepackage[latin1]{inputenc}

\usepackage{amssymb,color}
\usepackage{amsthm,amsmath}

\usepackage{upgreek}

\newtheorem{theo}{Theorem}[section]
\newtheorem{coro}[theo]{Corollary}

\newtheorem{pro}[theo]{Proposition}
\newtheorem{prob}[theo]{Problem}

\theoremstyle{definition}
\newtheorem{defin}[theo]{Definition}
\newtheorem{rem}[theo]{Remark}
\newtheorem{exam}[theo]{Example}

\def\cK{{{\mathcal K}}}

\def\N{\mathbb N}

\def\R{\mathbb R}

\def\eps{\varepsilon}
\def\epsilon{\eps}

\def\01^N{\{0,1\}^{n}}

\def\1^K{\{0,1\}^{k}}
\def\n+1pla{(i_1,\cdots,i_{n+1})}
\def\u^n+1{\{0,1\}^{n+1}}

\def\norm{\parallel \cdot \parallel}

\def\cero{0_X}
\def\ceroe{0_E}

\begin{document}

\begin{frontmatter}



\title{Denting Points of Convex Sets and Weak Property ($\pi$) of Cones in Locally Convex Spaces}


\author{Fernando García-Castaño\corref{cor1}}
\ead{Fernando.gc@ua.es} 
\cortext[cor1]{Corresponding author. Orcid: 0000-0002-8352-8235}
\author{M. A. Melguizo Padial}
\ead{ma.mp@ua.es}
\author{G. Parzanese}
\ead{gpzes@yahoo.it}
\address{Departamento de Matemática Aplicada\\ Escuela Politécnica
Superior\\ Universidad de Alicante\\ 03080 San Vicente del Raspeig, Alicante, Spain}

\begin{abstract}
In this paper we first extend from normed spaces to locally convex spaces some characterizations of denting points in convex sets. On the other hand, we also prove that in an infrabarreled locally convex space a point in a convex set is denting if and only if it is a point of continuity and an extreme point of the closure of such a convex set under the strong topology  in the second dual. The version for normed spaces of the former equivalence is new and contains, as particular cases, some known and remarkable results.  We also extend from  normed spaces to locally convex spaces some known characterizations of the weak property (\(\pi\)) of cones. Besides, we provide some new results regarding the angle property of cones and related. We also state that the class of cones in normed spaces having a pointed completion is the largest one for which the vertex is a denting point if and only if it is a point of continuity. Finally we analyse and answer several problems in the literature concerning geometric properties of cones which are related with density problems into vector optimization.
\end{abstract}


\end{frontmatter}


\section{Introduction}
Dentability was connected with the Radon-Nikod\'ym property in Banach spaces for the first time  in \cite{Rieffel:66:Irvine}. Later, the Radon-Nikod\'ym property has been studied into the context of locally convex spaces as can be seen, for example, in \cite{Chi1977,Egghe1978} and the references therein. A very remarkable result in the study of denting points is due to Lin-Lin-Troyanski in \cite{Lin1988}. This states that the notions of denting point and point of continuity become equivalent for extreme points in closed, convex, and bounded subsets of Banach spaces. The former result was generalized to normed spaces in \cite{GARCIACASTANO2018} as follows: denting points and points of continuity become equivalent for preserved extreme points in convex subsets of normed spaces. Let us note that a preserved extreme point is an extreme point of the closure under the weak star topology on the bidual space of a convex set. In Section \ref{Section_Convex_sets} we study denting points of convex sets in locally convex spaces. In order to generalize the results in \cite{GARCIACASTANO2018}  to locally convex spaces we have to introduce the notion of bounded denting point (which is stronger than denting point). This strengthening of the notions involved (stated in Definitions \ref{Defi_DentingPoint} and \ref{Defi_PointContinuity_y_boundedPC}) is, in some way, natural because working in normed spaces some things happen automatically but some of them generally fail in locally convex spaces --for example, in normed spaces bounded linear transformations are always continuous--. In Theorem~\ref{Teo_Generalizacion_dentabilidad_convexos_RACSAM} we prove (among other equivalences) that bounded denting points and bounded points of continuity become equivalent for preserved extreme points in convex subsets of locally convex spaces.  For infrabarreled locally convex spaces we state in Theorem~\ref{Teo_dentabilidad_convexos_nuevo} that bounded denting points and bounded points of continuity become equivalent for extreme points in the closure under the strong topology on the second dual of a convex set. An advantage of Theorem~\ref{Teo_dentabilidad_convexos_nuevo} regarding Theorem~\ref{Teo_Generalizacion_dentabilidad_convexos_RACSAM} is that the closure under the strong topology is smalller than the closure under the weak star topology. Furthermore, the version of Theorem~\ref{Teo_dentabilidad_convexos_nuevo} for normed spaces states that denting points and points of continuity become equivalent for extreme points in the completion of a convex set (Corollary~\ref{Tma_dentability_convex_nuevo}). Such a characterization is not contained in \cite{GARCIACASTANO2018} and it can be viewed as the natural generalization for normed spaces of the above mentioned characterization for denting points in Banach spaces  of Lin-Lin-Troyanski.

A cone is a particular type of convex set which plays an important role in vector optimization, operator equations, etc. Some properties of cones can be characterized by geometric properties of their vertex. Using the Lin-Lin-Troyanski's characterization above, Daniilidis proved in \cite{Daniilidis2000} that the vertex of a cone is denting if and only if it is a point of continuity for any pointed and closed cone in a Banach space. In \cite[page 629]{Gong1995}, Gong stated a new density result in vector optimization and asked if the condition that the vertex is a point of continuity for a pointed and closed cone in a normed space is really weaker than the vertex is a denting point for such a cone. On the other hand, in \cite{GARCIACASTANO20151178,GARCIACASTANO2018} some properties of cones in normed spaces were characterized by the property of the vertex being a denting point; some of them generalized Daniilidis characterization and also  Theorem 4 in \cite{Kountzakis2006}. In \cite{GARCIACASTANO20151178}, the authors stated that the vertex of a cone is a denting point if and only if it is a point of continuity and the closure of the cone under the weak star topology is pointed; in \cite{GARCIACASTANO2018} it was stated that the vertex of a cone is a denting point if and only if it is a point of continuity and the span of the dual cone is a dense subspace. In Section~\ref{Section_cones_LCS} of this paper we generalize to locally convex spaces the characterizations  for cones in \cite{GARCIACASTANO20151178,GARCIACASTANO2018} giving rise to new characterizations of the weak property \((\pi)\). For example, Theorem~\ref{Teo_Generalizacion_JMAA_RACSAM_lcs} states (among other equivalences) that a cone in a locally convex space satisfies the weak property \((\pi)\) if and only if the vertex is  a bounded point of continuity and the bidual cone is pointed if and only if the vertex is  a bounded point of continuity and the span of the dual cone is a dense subspace under the strong topology. It is worth pointing out that some characterizations in \cite{GARCIACASTANO2018} hold true in locally convex spaces under extra assumptions, as Theorem \ref{Teo_Condi_RACSAM_sufic_elc} shows. For infrabarreled spaces we provide Theorem \ref{Teo_weak_property_completitud_pointed_lcs} which states that a cone satisfies the weak property \((\pi)\) if and only if the vertex is  a bounded point of continuity and the closure of the cone under the strong topology on the second dual is pointed. On the other hand, Theorem \ref{Teo_Caract_strong_angle_property} provides new characterizations for cones in infrabarreled spaces enjoying the strong angle property, which completes those stated in Theorem 2.2 of \cite{Qiu2001}. We also obtain some dual results, in particular Theorem~\ref{Teo_Caract_strong_angle_property_dual} (resp. \ref{Tma_carac_dual_cone_barreled_elc} and \ref{Tma_neces_dual_cone_barreled_elc}) is the dual version of Theorem~\ref{Teo_Caract_strong_angle_property} (resp. \ref{Teo_Generalizacion_JMAA_RACSAM_lcs} and \ref{Teo_Condi_RACSAM_sufic_elc}).

In Section \ref{Section_cones_normed_spaces} we provide Theorem \ref{Teo_denting_equivalente_completitud_cone}, which is the version for normed spaces of Theorem \ref{Teo_weak_property_completitud_pointed_lcs}, and it states that the vertex of a cone is a denting point if and only if it is point of continuity and the completion of the cone is pointed. The former can be viewed as the natural generalization for normed spaces of Daniilidis characterization in \cite{Daniilidis2000}. On the other hand, Remark \ref{Remark_Clase_cones_mas_grande_PC_equiv_DP} states that the class of cones having a pointed completion is the largest one for which the vertex is a denting point if and only if it is a point of continuity. The former means that the class of cones having a pointed completion is the largest one for which Gong's question in \cite[page 629]{Gong1995} has a negative answer. Making use of an example in \cite{SONG2003308} we answer Gong's problem above mentioned and problems  \cite[page 624]{Gong1995} and \cite[Problem 1.7]{GARCIACASTANO20151178}.

\section{Denting points in convex sets}\label{Section_Convex_sets}
In this section we state some characterizations of denting points in convex sets in terms of the notion of point of continuity. In the first part we analyse how to extend from normed spaces to locally convex spaces some known results. Later, we state a characterization regarding denting points of convex sets in infrabarreled locally convex spaces whose corresponding version for normed spaces provides a new characterization which contains, as particular cases, some classic results in the literature.

Let $(E,\tau)$ denote a locally convex space (l.c.s.) and $x\in E$, we will denote by $\tau(x)$ the family of neighbourhoods of $x$ under the topology $\tau$ on $E$. To simplify the notation, we write $E$ instead of $(E,\tau)$ omitting the original topology $\tau$ unless it is indispensable to avoid confusion.  In this work, we only consider Hausdorff locally convex spaces. By $E^*$ we denote the dual of $E$ and $\sigma(E,E^*)$ denotes the weak topology on $E$. 
\begin{defin}\label{Defi_DentingPoint}
Let $E$ be a l.c.s. and $C\subset E$ a convex subset. 
\item[(i)] We say that $x\in C$ is a denting point of $C$, written $x \in \mbox{DP}(C)$, if for every $U \in \tau(x)$ we have $x \not \in \overline{\mbox{conv}}{(C\setminus U)}$.
\item[(ii)] We say that $x\in C$ is a bounded denting point of $C$, written $x \in \mbox{bDP}(C)$, if $x \in \mbox{DP}(C)$ and there exists $W\in \sigma(E,E^*)(x)$ such that $W\cap C$ is bounded.
\end{defin}
Given $x^* \in E^*$ and $\lambda \in \R$ we denote by $\{x^*<\lambda\}$ the following half-space in $E$, $\{x \in E\colon x^*(x)<\lambda\}$. Let $E$ be a l.c.s, it is clear that $x \in \mbox{DP}(C)$  if and only if for every $U \in \tau(x)$ there exist $x^* \in E^*$ and $\lambda \in \R$ such that $x \in \{x^*<\lambda\}\cap C \subset U$. Any set of the form $\{x^*<\lambda\}\cap C$ is called a $\sigma(E,E^*)$-slice of $C$.

In the following definition we extend the notion of strongly extreme point from normed spaces (see \cite{Guirao2014}) to locally convex spaces. 
\begin{defin}
Let $E$ be a l.c.s. and $C\subset E$ a convex subset. We say that $x\in C$ is a strongly extreme point  of $C$, written $x \in \mbox{s-ext}(C)$, if given two nets $(c_{\gamma})_{\gamma \in \Gamma}$ and $(c_{\gamma}')_{\gamma\in \Gamma}$ in $C$ such that $\lim_{\gamma}\frac{c_{\gamma}+c_{\gamma}'}{2}=x$, then $\lim_{\gamma}c_{\gamma}=x$. If $\lim_{\gamma}c_{\gamma}=x$ is taken under the topology $\sigma(E,E^*)$,  then we say that $x$ is a $\sigma(E,E^*)$-strongly extreme point  of $C$, written $x \in \sigma(E,E^*)\mbox{-s-ext}(C)$.
\end{defin}

\begin{pro}\label{Prop_denting_implies_strongly_extreme_convex}
Let $E$ be a l.c.s. and $C \subset E$ a convex subset. Then $\mbox{DP}(C)\subset \mbox{s-ext}(C)$. 
\end{pro}
\begin{proof}
Assume $x=\ceroe \in \mbox{DP}(C)$. Let us fix two nets $(c_{\gamma})_{\gamma\in \Gamma}$, $(c_{\gamma}')_{\gamma\in \Gamma}$ in $C$ such that $\lim_{\gamma} \frac{c_{\gamma}+c_{\gamma}'}{2}=\ceroe$, we will check that $\lim_{\gamma}c_{\gamma}=\ceroe$. Assume the contrary. Hence there exists $U\in \tau(\ceroe) $ and  a subnet, $(c_{\gamma_{\beta}})_{\beta\in \Gamma'}$, of $(c_{\gamma})_{\gamma\in \Gamma}$ such that $\frac{c_{\gamma_{\beta}}}{2} \not \in U$, $\forall \beta\in \Gamma'$. Now, fix $W\in\tau(\ceroe)$ closed and balanced such that $W-W\subset U$. By assumption, there exists $x^* \in E^*$ and $\lambda \in \R$ such that $\ceroe \in \{x^*\leq \lambda \}\cap C \subset W$. It is not restrictive to assume that $\frac{c_{\gamma_{\beta}}+c_{\gamma_{\beta}}'}{2} \in \{x^*\leq \lambda\}\cap C$, $\forall \beta\in \Gamma'$. Now, since $\frac{c_{\gamma_{\beta}}}{2} \not \in  \{x^*\leq \lambda\}\cap C$ and $\frac{c_{\gamma_{\beta}}+c_{\gamma_{\beta}}'}{4} \in  \{x^*\leq \lambda\}\cap C$,  $\forall \beta\in \Gamma'$, by convexity, we have $\frac{c_{\gamma_{\beta}}'}{2} \in \{x^*\leq \lambda\}\cap C \subset W$, $\forall \beta\in \Gamma'$. As a consequence, we have $\frac{c_{\gamma_{\beta}}}{2}=\frac{c_{\gamma_{\beta}}+c_{\gamma_{\beta}}'}{2}-\frac{c_{\gamma_{\beta}}'}{2} \in \{x^*\leq \lambda\}\cap C-W\subset W-W \subset U$, a contradiction.

In case  $x\not =\ceroe$, we have $\ceroe \in \mbox{DP}(C-x)$. By the former paragraph $\ceroe \in \mbox{s-ext}(C-x)$, and clearly $x \in \mbox{s-ext}(C)$.
\end{proof}
Given a l.c.s. $E$, we consider its dual $E^*$ topologized under the strong topology $\beta(E^*,E)$, i.e., the topology on $E^*$ of uniform convergence on bounded sets in $E$. We denote by $E^{**}$ the dual of  $(E^*, \beta(E^*,E))$, which is called the second dual. On $E^{**}$ we will consider two topologies, the weak-star topology, $\sigma(E^{**},E^*)$, and the strong topology, $\beta(E^{**},E^*)$. It is worth pointing out that $\sigma(E^{**},E^*)$ is compatible with the dual par $<E^*,E^{**}>$ whereas $\beta(E^{**},E^*)$ is not (in general). 

Let us now consider the canonical mapping \(J_{E}:E \rightarrow E^{**}\), defined by $J_{E}(x)(x^*):=x^*(x)$, $\forall x^* \in E^*$, which is linear and one-to-one. From now, we identify each $J_{E}(x)\in E^{**}$ with its corresponding $x \in E$. In this way, given a set $A\subset E$ we will denote by $\widetilde{A}$ the closure of $J_{E}(A)$ in $E^{**}$  under  $\sigma(E^{**},E^*)$ (i.e., $\overline{J_{E}(A)}^{\sigma(E^{**},E^*)}$), and by $\widehat{A}$ the closure of $J_{E}(A)$ in $E^{**}$  under $\beta(E^{**},E^*)$ (i.e., $\overline{J_{E}(A)}^{\beta(E^{**},E^*)}$); clearly $\widehat{A}\subset \widetilde{A}$ and $ \widetilde{A}\cap E=\overline{A}^{\sigma(E,E^*)}$. Since in large expressions the former notation may be confusing, sometimes we may write $\overline{A}^{\sigma(E^{**},E^*)}$ instead of $\overline{J_{E}(A)}^{\sigma(E^{**},E^*)}$, and  $\overline{A}^{\beta(E^{**},E^*)}$ instead of $\overline{J_{E}(A)}^{\beta(E^{**},E^*)}$. It is worth pointing out that $J_{E}$ is $\tau$-$\sigma(E^{**},E^*)$ continuous and a $\sigma(E,E^*)$-$\sigma(E^{**},E^*)$ homeomorphism from $E$ onto $J_{E}(E)$ (details in \cite[Theorem~5.10]{Osborne}). 

\begin{defin}\label{Defi_extreme_point}
Let $E$ be a linear space and $C\subset E$ a convex subset. We say that $x \in C$ is an extreme point of $C$, written $x\in$ ext$(C)$, if the equality $x=\frac{c_1+c_2}{2}\in C$  implies $x=c_1=c_2$, for any $c_1$, $c_2\in C$. 
\end{defin}
It is clear that $\sigma(E,E^*)\mbox{-s-ext}(C)\subset $ ext$(C)$. Such an inclusion remains true under the $\sigma(E^{**},E^*)$-closure of $C$, as the following result states. Let us remind that given a subset $A \subset E$, we denote by conv$(A)$ to the convex hull of $A$.
\begin{pro}\label{Prop_weakly_strongly_extreme_implies_Convex_tilde_extreme}
Let $E$ be a l.c.s. and $C\subset E$ a convex subset. Then $\sigma(E,E^*)\mbox{-s-ext}(C)\subset \mbox{ext}(\widetilde{C})$.
\end{pro}
\begin{proof}
Since $x \in \sigma(E,E^*)\mbox{-s-ext}(C)$, given any $x^* \in E^*$ and every $\epsilon>0$ there exists $U(x^*,\epsilon)\in \tau(x)$ such that if $\{y_1,\,y_2\}\subset C$ and $\frac{y_1+y_2}{2}\in U(x^*,\epsilon)$, then  $\max\{|x^*(y_1-x)|,|x^*(y_2-x)|\}<\epsilon$. Assume that $x \not \in \mbox{ext}(\widetilde{C})$. Then there exists $x^{**}_1,\,x^{**}_2 \in \widetilde{C}$, $x^{**}_1\not = x^{**}_2$, such that $x= \frac{x^{**}_1+x^{**}_2}{2}$. Let us fix $x^*_0 \in E^*$, $\alpha$, $\beta \in \R$ such that $x^{**}_1(x^*_0)<\alpha<\beta<x^{**}_2(x^*_0)$. We consider a net $(z^1_{\gamma})_{\gamma \in \Gamma} \subset C$ which $\sigma(E^{**},E^*)$-converges to $x^{**}_1$ and a net $(z^2_{\gamma})_{\gamma \in \Gamma} \subset C$ which $\sigma(E^{**},E^*)$-converges to $x^{**}_2$. It is not restrictive to assume that $x^*_0(z^1_{\gamma})<\alpha<\beta<x^*_0(z^2_{\gamma})$, $\forall \gamma \in \Gamma$. Since all locally convex topologies consistent with a given dual pair have the same collection of closed convex sets, it follows that $x \in \overline{\{\frac{z^1_{\gamma}+z^2_{\gamma}}{2}\colon \gamma \in \Gamma\}}^{\sigma(E,E^*)}\subset \overline{conv}^{\sigma(E,E^*)}\{\frac{z^1_{\gamma}+z^2_{\gamma}}{2}\colon \gamma \in \Gamma\}=\overline{conv}\{\frac{z^1_{\gamma}+z^2_{\gamma}}{2}\colon \gamma \in \Gamma\}$. Now define $\epsilon_0:=\frac{\beta-\alpha}{2}>0$. Consider $n \geq 1$ and $\{\lambda_1,\ldots,\lambda_n\}\subset [0,1]$ such that $\sum_{j=1}^{n}\lambda_j=1$ and $\sum_{j=1}^{n}\frac{\lambda_j}{2}(z^1_{\gamma_j}+z^2_{\gamma_j})\in U(x^*_0,\epsilon_0)$, where $U(x^*_0,\epsilon_0)$ is given by the property of the beginning of the proof. We define $u_1:=\sum_{j=1}^{n}\lambda_jz^1_{\gamma_j}\in C$ and $u_2:=\sum_{j=1}^{n}\lambda_jz^2_{\gamma_j}\in C$. Since $\frac{u_1+u_2}{2}\in U(x^*_0,\epsilon_0)$, it follows that $\max\{|x^*(u_1-x)|,|x^*(u_2-x)|\}<\epsilon_0$, which implies $|x^*(u_1)-x^*(u_2)|\leq 2\epsilon_0$. On the other hand $x^*(u_1)=\sum_{j=1}^{n}\lambda_jx^*(z^1_{\gamma_j})<\alpha<\beta<\sum_{j=1}^{n}\lambda_jx^*(z^2_{\gamma_j})=x^*(u_2)$, which yields to $|x^*(u_2)-x^*(u_1)|=x^*(u_2)-x^*(u_1)>\beta-\alpha=2\epsilon_0$, a contradiction.
\end{proof}
Now we introduce the notion of point of continuity, which is weaker than the notion of denting point.
\begin{defin}\label{Defi_PointContinuity_y_boundedPC}
Let $E$ be a l.c.s. and $C\subset E$ a convex subset. 
\item[(i)] We say that $x$ is a point of continuity of $C$, written $x \in \mbox{PC}(C)$, if for every $U \in \tau(x)$, there exists $W \in \sigma(E,E^*)(x)$ such that $x \in W\cap C \subset U$.
\item[(ii)] We say that $x$ is a bounded point of continuity of $C$, written $x \in \mbox{bPC}(C)$, if $x \in \mbox{PC}(C)$ and there exists $W\in \sigma(E,E^*)(x)$ such that $W\cap C$ is bounded. 
\end{defin}
If either $C$ is bounded or $E$ is a normed space, then every point of continuity is a bounded point of continuity. On the other hand, in the following results we will show that extreme points connect denting points with points of continuity in a similar way, although not exactly the same, as it happens in normed spaces. The first one is a consequence of Propositions \ref{Prop_denting_implies_strongly_extreme_convex} and \ref{Prop_weakly_strongly_extreme_implies_Convex_tilde_extreme}.
\begin{theo}\label{Teo_dentability_ELC_widetildeC}
Let $E$ be a l.c.s. and $C \subset E$ a convex subset. Then $\mbox{DP}(C)\subset \mbox{PC}(C)\cap \mbox{ext}(\widetilde{C})$.
\end{theo}
In order to establish the reverse inclusion of the former result we will need some preliminary results and new terminology. Given $x^{*} \in E^*$ and $\lambda \in \R$, we will denote by $\{x^{*}\leq \lambda\}^{**}$ the set in $E^{**}$ defined by $\{x^{**} \in E^{**} \colon x^{**}(x^{*})<\lambda\}$. Let $A\subset E^{**}$, a  $\sigma(E^{**},E^*)$-open slice of $A$ is a set of the form $\{x^{*}<\lambda\}^{**}\cap A$ for some $x^{*} \in E^*$ and $\lambda \in \R$.
\begin{pro}[Choquet's Lemma]\label{Lemma_Choquet_LCS}
Let $E$ be a l.c.s., $C \subset E^{**}$ a subset $\sigma(E^{**},E^*)$-compact and convex, and $x^{**} \in \mbox{ext}(C)$. Then the family of $\sigma(E^{**},E^*)$-open slices of $C$ containing $x^{**}$ forms a neighbourhood base of $x^{**}$ for the topology $\sigma(E^{**},E^*)$  relative  to $C$.
\end{pro}
\begin{proof}
Let $W=\cap_{i=1}^n\{x^{*}_i < \alpha_i\}^{**}$ be a $\sigma(E^{**},E^*)$ neighbourhood of $x^{**}$, for some $n \geq 1$, $x^{*}_i \in E^*$, and $\alpha_i \in \R$. For every $i \in \{1,\ldots,n\}$ we define $K_i:=C\cap \{x^*_i \geq \alpha_i\}^{**}\subset C$ which is $\sigma(E^{**},E^*)$-compact and convex. Denote $D:=\cup_{i=1}^nK_i$. Clearly $x^{**} \not \in D$. Since $x^{**}\in \mbox{ext}(C)$ and $D\subset C$, we have $x^{**} \not \in \mbox{conv}(D)$. Indeed, assume that $x^{**}=\sum_{j=1}^{m}\lambda_j d^{**}_j$ for some $m>2$, $d^{**}_j\in D\subset C$, and $\lambda_j \geq 0$ for every $j=1,\ldots,m$ such that $\sum_{j=1}^{m}\lambda_j=1$. We define $y^{**}:=\sum_{j=1}^{m-1}\frac{\lambda_j}{\sum_{j=1}^{m-1}\lambda_j} d^{**}_j\in \mbox{conv}(D)\subset C$. We can write $x^{**}=\sum_{j=1}^{m-1}\lambda_jy^{**}+\lambda_md^{**}_m$. Since $x^{**}\in \mbox{ext}(C)$, it follows that $x^{**}=y^{**}=d^{**}_m\in D$, a contradiction.  As $\mbox{conv}(D)=\mbox{conv}(\cup_{i=1}^nK_i)$ is the convex hull of a finite number of $\sigma(E^{**},E^*)$-compact convex subsets we have that $\mbox{conv}(D)$ is also $\sigma(E^{**},E^*)$-compact. As a consequence, $x^{**} \not \in \overline{\mbox{conv}}^{\sigma(E^{**},E^*)}(D)$. Let us check that $\mbox{conv}(D)$ is $\sigma(E^{**},E^*)$-compact. An argument similar to that used in the proof of $x^{**} \not \in \mbox{conv}(D)$ yields $\mbox{conv}(D)=\{\sum_{i=1}^{n}\lambda_i k_i\colon \lambda_i \geq 0,\,k_i \in K_i, \mbox{ for } i=1,\ldots,n,\,\sum_{i=1}^{n}\lambda_i=1\}$. Let us note that the set $K:=\{(\lambda_1,\ldots,\lambda_n)\in \R^n\colon \lambda_i \geq 0\mbox{ for } i=1,\ldots,n \mbox{ and }\sum_{i=1}^{n}\lambda_i=1\}$ is a compact subset of $\R^n$. Now, fix a net $(x_{\gamma})_{\gamma \in \Gamma}\subset \mbox{conv}(D)$, we will find a convergent subnet. For every $\gamma \in \Gamma$, we write $x_{\gamma}=\sum_{i=1}^{n}\lambda^{\gamma}_ik^{\gamma}_i$, for $(\lambda^{\gamma}_1,\ldots,\lambda^{\gamma}_n)\in K$ and $k^{\gamma}_i\in K_i$, $\forall i\in \{1,\ldots,n\}$. Then $((\lambda^{\gamma}_1,\ldots,\lambda^{\gamma}_n))_{\gamma \in \Gamma}\subset K$ is a net which has subnet $((\lambda^{\gamma_{\beta}}_1,\ldots,\lambda^{\gamma_{\beta}}_n))_{\beta \in \Gamma'}\subset K$ converging to some $(\lambda_1,\ldots,\lambda_n)\in K$. Now, for every $\beta \in \Gamma'$ we write $x_{\gamma_{\beta}}=\sum_{i=1}^{n}\lambda^{\gamma_{\beta}}_ik^{\gamma_{\beta}}_i$. By the compactness of each $K_i$, it follows that each net $(k^{\gamma_{\beta}}_i)_{\beta \in \Gamma'}\subset K_i$ has a convergent subnet. For simplicity of the notation, we will assume that for every $i \in \{1,\ldots,n\}$, there exists $k_i\in K_i$ such that $\lim_{\beta}k^{\gamma_{\beta}}_i=k_i\in K_i$. From the above it follows that the net  $(x_{\gamma_{\beta}})_{\beta \in \Gamma'}$ converges to $\sum_{i=1}^n\lambda_ik_i \in \mbox{conv}(D)$ proving the compactness of $\mbox{conv}(D)$. Now, since $x^{**} \not \in \overline{\mbox{conv}}^{\sigma(E^{**},E^*)}(D)$, we apply \cite[Theorem 3.32]{Montensinos2ndbook} to find $x^{*} \in E^*$ and $\lambda \in \R$ such that $x^{**}(x^{*})<\lambda<\sup\{y^{**}(x^{*})\colon y^{**} \in 
\overline{\mbox{conv}}^{\sigma(E^{**},E^*)}(D)\}$. Then $x^{**} \in C\cap \{x^{*} < \lambda\}^{**}\subset C \cap W$.
\end{proof}
Let $A\subset E$ be a subset and $r\in \R$, we define $rA:=\{ra\colon a \in A\}$. Let us fix a dual pair $<E_1,E_2>$ and $A\subset E_1$ a subset, we denote the (absolute) polar of $A$  by $A^{\circ}:=\{f\in E_2\colon |f(x)|\leq 1,\, \forall x \in A\}$. Let us note that the construction of polars is fundamental in describing the collection of all consistent topologies for a dual pair.
\begin{pro}\label{Prop_pseudoslices_base_LCS_convex_set}
Let $E$ be a l.c.s. and $C\subset E$ a convex subset. If $x\in \mbox{bPC}(C)\cap \mbox{ext}(\widetilde{C})$, then there exists  $U \in \tau(\ceroe)$ such that for every $0 <r\leq 1$: 
\item[(i)] $C\cap (rU+x)$ is bounded.
\item[(ii)] The family of open slices of $C\cap (rU+x)$ containing $x$ forms a neighbourhood base of $x$ for the topology $\sigma(E,E^*)$ relative to $C\cap (rU+x)$.
\end{pro}
\begin{proof}
First, we assume that $x=\ceroe \in \mbox{bPC}(C)\cap \mbox{ext}(\widetilde{C})$. It is sufficient to prove the case $r=1$. 

(i) As $\ceroe \in \mbox{bPC}(C)$, there exists $W\in \sigma(E,E^*)(\ceroe)$ such that $C \cap W$ is bounded. On the other hand, since the initial topology $\tau$ on $E$ is finer than $\sigma(E,E^*)$, there exists $U \in \tau(\ceroe)$ such that $C\cap U\subset C\cap W$. Hence $C\cap U$ is also bounded.

(ii) The proof will be divided into two parts. In the first one we will check that for $U \in \tau(\ceroe)$ from (i), the set $\overline{C \cap U}^{\sigma(E^{**},E^*)}$ is a $\sigma(E^{**},E^*)$ compact set in $E^{**}$.  Since $C\cap U$ is bounded, it follows that the corresponding (absolute) polar $(C \cap U)^{\circ} \in \beta(E^*,E)(\ceroe)$. Applying again polars (now under the duality $<E^*,E^{**}>$), we have $((C \cap U)^{\circ})^{\circ}\subset E^{**}$ which is a $\sigma(E^{**},E^*)$ compact set (by the Alaoglu's Compactness Theorem, \cite[Theorem 5.105]{ali-bor}). Since $C \cap U\subset ((C \cap U)^{\circ})^{\circ}$, we conclude that $\overline{C \cap U}^{\sigma(E^{**},E^*)}$ is a $\sigma(E^{**},E^*)$ compact set.
In the second part, we fix some $W\in \sigma(E,E^*)(\ceroe)$ such that $W=\cap_{i=1}^n\{x^{*}_i\leq \alpha_i\}$ for some $n \geq 1$, $x^{*}_i \in E^*$, $\alpha_i \in \R$. Since $\ceroe \in \mbox{ext}(\widetilde{C})$, we have $\ceroe \in \mbox{ext}(\overline{C \cap U}^{\sigma(E^{**},E^*)})$ for  $U \in \tau(\ceroe)$ from (i). It is not restrictive to assume that $U$ is closed. Now, by Proposition \ref{Lemma_Choquet_LCS} there exist $x^{*} \in E^*$ and $\lambda \in \R$ such that $\ceroe\in \{x^{*}\leq \lambda\}^{**}\cap \overline{C \cap U}^{\sigma(E^{**},E^*)} \subset \overline{W\cap C \cap U}^{\sigma(E^{**},E^*)}$. Then $\{x^{*} \leq \lambda\}\cap C\cap U\subset \overline{W\cap C \cap U}^{\sigma(E^{**},E^*)}\cap E= \overline{W\cap C \cap U}^{\sigma(E,E^*)}\subset W\cap \overline{C \cap U}^{\sigma(E,E^*)}$. As a consequence, $\{x^{*} \leq \lambda\}\cap C\cap U\subset W\cap C\cap U$.

Let us assume now that $x\in \mbox{bPC}(C)\cap \mbox{ext}(\widetilde{C})$, $x\not =\ceroe$. Then $\ceroe\in \mbox{bPC}(C-x)\cap \mbox{ext}(\widetilde{C}-x)$. Applying the case $x=\ceroe$ already proved, there exists $U \in \tau(\ceroe)$ such that for every $0 <r\leq 1$ the set $(C-x)\cap (rU)$ is bounded and the family of open slices of $(C-x)\cap (rU)$ containing $\ceroe$ forms a neighbourhood base of $\ceroe$ for the topology $\sigma(E,E^*)$ relative to $(C-x)\cap (rU)$. It is a simple matter to check that $C\cap (rU+x)$ is bounded and that the family of open slices of $C\cap (rU+x)$ containing $x$ forms a neighbourhood base of $x$ for the topology $\sigma(E,E^*)$ relative to $C\cap (rU+x)$.
\end{proof}

The next theorem extends from normed spaces to locally convex spaces Proposition 9.1 in \cite{Guirao2014} and Proposition 2 in \cite{GARCIACASTANO2018}. 

\begin{theo}\label{Teo_caract_extremal_C_tilde} 
Let $E$ be a l.c.s., $C \subset E$ a convex subset, and $x \in C$. Assume that there exists $U \in \tau(x)$ such that $C\cap U$ is bounded and convex. The following are equivalent.
\item[(i)] $x \in \sigma(E,E^*)\mbox{-s-ext}(C\cap U)$.
\item[(ii)] $x\in \mbox{ext}(\widetilde{C})$.
\item[(iii)] The family of open slices of $C\cap U$ containing $x$ forms a neighbourhood base of $x$ for the topology $\sigma(E,E^*)$ relative to $C\cap U$.
\end{theo}
\begin{proof}
(i)$\Rightarrow$(ii)  The proof of Proposition \ref{Prop_weakly_strongly_extreme_implies_Convex_tilde_extreme} can be easily adapted changing the role of $C$ there by $C\cap U$ now.

(ii)$\Rightarrow$(iii) Assertion (i) in Proposition \ref{Prop_pseudoslices_base_LCS_convex_set} is satisfied. Then, the proof of assertion (ii) in Proposition \ref{Prop_pseudoslices_base_LCS_convex_set} works here.

(iii)$\Rightarrow$(i) Let us consider $x\in C\cap U$, and two nets $(y_{\gamma})_{\gamma \in \Gamma}$ and $(z_{\gamma})_{\gamma \in \Gamma}$ in $C\cap U$ such that $\lim_{\gamma}\frac{y_{\gamma}+z_{\gamma}}{2}=x$. In order to show that $(y_{\gamma})_{\gamma \in \Gamma}$ converges to $x$ under $\sigma(E,E^*)$ relative to $C\cap U$ we will check that
\begin{equation}\label{Eq_Tma_equiv_dual_cone_pointed_convex}
x\in \overline{\{y_{\gamma_{\beta}}\colon \beta \in \Gamma'\}}^{(C\cap U,\,\sigma(E,E^*)|_{C\cap U})} 
\end{equation}
for every subnet $(y_{\gamma_{\beta}})_{\beta \in \Gamma'}$ of $(y_{\gamma})_{\gamma \in \Gamma}$. For simplicity of notation we will prove (\ref{Eq_Tma_equiv_dual_cone_pointed_convex}) for the initial net $(y_{\gamma})_{\gamma \in \Gamma}$.  Assume that (\ref{Eq_Tma_equiv_dual_cone_pointed_convex}) is not true for $(y_{\gamma})_{\gamma \in \Gamma}$. Then, there exist $x^* \in E^*$ and $\alpha \in \R$ such that $x^*(x)<\alpha$ and $y_{\gamma}\not \in \{x^*<\alpha\}\cap C\cap U$, $\forall \gamma \in \Gamma$. We claim that $(z_{\gamma})_{\gamma}$ converges to $x$ under the topology $\sigma(E,E^*)$ relative to $C\cap U$. Indeed, fix arbitrary $y^* \in E^*$ and $\lambda \in \R$ such that $x \in V:=\{y^*<\lambda\}\cap\{x^*<\alpha\}\cap C \cap U$. Then, there exists $\gamma_0 \in \Gamma$ such that $\frac{y_{\gamma}+z_{\gamma}}{2} \in V$, $\forall \gamma \geq \gamma_0$. By convexity $z_{\gamma}\in V \subset \{y^*<\lambda\}\cap C \cap U$, $\forall \gamma \geq \gamma_0$. Hence $(z_{\gamma})_{\gamma}$, and consequently $(y_{\gamma})_{\gamma}$, converges to $x$ under the topology $\sigma(E,E^*)$ relative to $C\cap U$. This fact contradicts that $y_{\gamma}\not \in \{x^*<\alpha\}\cap C\cap U$, $\forall \gamma \in \Gamma$. Therefore (\ref{Eq_Tma_equiv_dual_cone_pointed_convex}) is true for $(y_{\gamma})_{\gamma}$.
\end{proof}

In the following result we strengthen the assumptions in Theorem \ref{Teo_dentability_ELC_widetildeC} to provide the reverse inclusion. Let us recall that a subset $A$ in a l.c.s. $E$ is said to be balanced if $\alpha A \subset A$, for every $\alpha \in [-1,1]$.
\begin{theo}\label{Teo_dentability_ELC_cond_fuertePC_convex}
Let $E$ be a l.c.s. and $C\subset E$ a convex subset. Then $\mbox{bPC}(C)\cap \mbox{ext}(\widetilde{C})\subset \mbox{bDP}(C)$.
\end{theo}
\begin{proof}
We first consider the case $x=\ceroe\in \mbox{bPC}(C)\cap \mbox{ext}(\widetilde{C})$. We fix $U \in \tau(\ceroe)$ given by Proposition~\ref{Prop_pseudoslices_base_LCS_convex_set}. It is not restrictive to assume that $U$ is balanced. By assumption there exists $W\in \sigma(E,E^*)(\ceroe)$ such that  $W\cap C$ is bounded. Then, we can choose $\rho>0$ such that $W\cap C \subset \rho U$. If $\rho \leq \frac{1}{4}$, then $W\cap C \subset \frac{1}{4}U$. Otherwise, we have $(\frac{1}{4\rho}W)\cap \frac{1}{4\rho}C \subset \frac{1}{4}U$, being $\frac{1}{4\rho}<1$. Since $\frac{1}{4\rho}W\in \sigma(E,E^*)(\ceroe)$  and $\ceroe\in \mbox{bPC}(\frac{1}{4\rho}C)\cap \mbox{ext}(\frac{1}{4\rho}\widetilde{C})$, we can apply assertion (ii) in Proposition~\ref{Prop_pseudoslices_base_LCS_convex_set}  to $(\frac{1}{4\rho}W)\cap \frac{1}{4\rho}C$ with the same neighbourhood $U$ and $r=\frac{1}{2}$. However, since $\ceroe \in (\frac{1}{4\rho}W)\cap \frac{1}{4\rho}C \subset  W\cap C$ we assume (for simplicity of the notation) that $W\cap C \subset \frac{1}{4}U$ and apply assertion (ii) in Proposition~\ref{Prop_pseudoslices_base_LCS_convex_set}  to $W\cap C$  with the former $U$ and $r=\frac{1}{2}$. Now, we pick  $x^* \in E^*$ and $\lambda > 0$ such that $\ceroe\in\{x^*<\lambda\}\cap C \cap \frac{1}{2}U \subset W\cap C \cap \frac{1}{2}U \subset \frac{1}{4}U$. We claim that $\{x^*<\lambda\}\cap C \subset \frac{1}{2}U\cap C$. Assume the contrary. Then there exists some $x \in C$, $x\not =\ceroe$, such that $x^*(x) <\lambda$ and $x \not \in \frac{1}{2} U$. Consider $r:=\inf\{s>0\colon x \in sU\}$. Then $r>\frac{1}{2}$. Next, we will check that $y=\frac{3}{8r}x\in \{x^*<\lambda\}\cap C \cap  \frac{1}{2}U\setminus \frac{1}{4}U$, the contradiction we were looking for. On the one hand, $x^*(y)=\frac{3}{8r}x^*(x)<\frac{3}{4}\lambda<\lambda$. On the other hand, $y=\frac{1}{2}\frac{x}{\frac{4}{3}r}$ and \(x\in \frac{4}{3}rU\), hence $y \in \frac{1}{2}U$. Finally, $y=\frac{3}{8r}x \in \frac{1}{4} U$ if and only if $\frac{3}{2r}x=\frac{x}{\frac{2}{3}r} \in U$, but  $x \not \in \frac{2}{3}r U$ by the minimality of $r$.

Now we consider the case $x\in \mbox{bPC}(C)\cap \mbox{ext}(\widetilde{C})$, $x\not =\ceroe$. It is clear that $\ceroe\in \mbox{bPC}(C-x)\cap \mbox{ext}(\widetilde{C}-x)$. By the first part of the proof we have $\ceroe \in \mbox{bDP}(C-x)$. Hence, $x \in \mbox{bDP}(C)$.
\end{proof}
The following result  generalizes  \cite[Theorem 4]{GARCIACASTANO2018} to locally convex spaces. This is a consequence of Propositions \ref{Prop_denting_implies_strongly_extreme_convex} and \ref{Prop_weakly_strongly_extreme_implies_Convex_tilde_extreme}, and Theorem~\ref{Teo_dentability_ELC_cond_fuertePC_convex}. 
\begin{theo}\label{Teo_Generalizacion_dentabilidad_convexos_RACSAM}
Let $E$ be a l.c.s. and $C \subset E$ a convex subset. Then
$\mbox{bDP}(C)=\mbox{bPC}(C)\cap  \sigma(E,E^*)\mbox{-s-ext}(C)=\mbox{bPC}(C)\cap  \mbox{ext}(\widetilde{C})$.
\end{theo}

The notions involved in the former characterization can be relaxed when dealing with bounded sets.
\begin{coro}\label{Coro_Generalizacion_dentabilidad_convexos_acotados_RACSAM}
Let $E$ be a l.c.s. and $C \subset E$ a bounded convex subset. Then
$\mbox{DP}(C)=\mbox{PC}(C)\cap  \sigma(E,E^*)\mbox{-s-ext}(C)=\mbox{PC}(C)\cap  \mbox{ext}(\widetilde{C})$.
\end{coro}

Let us recall that $\widehat{C}$ denotes the closure of $J_{E}(C)$ (the image of $C$ under the map $J_{E}:E \rightarrow E^{**}$) in $E^{**}$ under the topology $\beta(E^{**},E^*)$. Next, a new result which makes use of the set $\widehat{C}$.

\begin{theo}\label{Teo_dentability_ELC_C_convex}
Let $E$ be a l.c.s. and $C \subset E$ a convex subset. Then $DP(C)\subset PC(C)\cap \mbox{ext}(\widehat{C})$.
\end{theo}
\begin{proof}
Propositions \ref{Prop_denting_implies_strongly_extreme_convex} and \ref{Prop_weakly_strongly_extreme_implies_Convex_tilde_extreme} yield $x \in \mbox{ext}(\widetilde{C})$. Now, since $x \in \widehat{C}\subset \widetilde{C}$, it follows that $x\in \mbox{ext}(\widehat{C})$.
\end{proof}
The last part of this section is devoted to transform the former inclusion in Theorem~\ref{Teo_dentability_ELC_C_convex} into an equality changing the sets of denting points and points of continuity by the sets of bounded denting points and bounded points of continuity. For such a purpose, we have to stablish some previous results and introduce some new terminology. 

\begin{defin}\label{Defi_weak_star_PointContinuity_convex_set}
Let $E$ be a l.c.s. and $C\subset E^{**}$ a convex subset. We say that $x^{**}$ is a $\sigma(E^{**},E^*)$ point of continuity of $C$, written $x^{**} \in \sigma(E^{**},E^*)\mbox{-PC}(C)$, if for every $U \in \beta(E^{**},E^*)(x^{**})$, there exists some $W$ in $\sigma(E^{**},E^*)(x^{**})$ such that $x^{**} \in W\cap C\subset U$.
\end{defin}

Next, we restrict our study to a particular class of locally convex spaces. Let us recall that a barrel in a l.c.s. is a closed, convex, balanced, and absorbent subset.
\begin{defin}\label{Defi_infra_barrelled}
We say that a l.c.s. $E$ is infrabarreled (resp. barreled) if each barrel in $E$ which absorbs all bounded sets is a neighborhood of $\ceroe$ (resp. if every barrel is a neighborhood of $\ceroe$).
\end{defin}
A barreled space is always infrabarreled but the converse is not always true. Fréchet spaces are barreled spaces and metrizable locally convex spaces are infrabarreled. Therefore, our next results (stated for infrabarreled spaces)  can be applied to normed spaces as a particular case. On the other hand, locally convex spaces of second category, or even first countable, are infrabarreled. So, LF-spaces and LB-spaces  are infrabarreled (\cite[Theorem~4.5 and Proposition 4.10]{Osborne}).
Let us note that if $E$ is infrabarreled, then $J_{E}$ is a $\tau$-$\beta(E^{**},E^*)$ homeomorphism onto $J_{E}(E)$ such that $J^{-1}_{E}(A^{\circ})=A_{\circ}$ for every bounded subset $A \subset E^*$, (see \cite[Theorem 5.10]{Osborne}).

\begin{pro}\label{Prop_PC_equivale_weak_star_PC_convex}
Let $E$ be an infrabarreled  l.c.s. and $C \subset E$ a convex subset. If $c\in C$,  then $c \in \mbox{PC}(C)$ if and only if $c \in \sigma(E^{**},E^*)\mbox{-PC}(\widetilde{C})$.
\end{pro}
\begin{proof}
For simplicity we only prove that $\ceroe \in \mbox{PC}(C)$ if and only if  $\ceroe \in \sigma(E^{**},E^*)\mbox{-PC}(\widetilde{C})$.

$\Leftarrow$ Let $U \in \tau(\ceroe)$. Since $J_{E}^{-1}$ is continuous, there exists $V \in \beta(E^{**},E^*)(\ceroe)$ such that $U=E\cap V$. Now, by assumption, there exists $W^{**}=\cap_{i=1}^n\{x^{*}_i<\lambda_i\}^{**} \in \sigma(E^{**},E^*)(\ceroe)$  such that $W^{**}\cap \widetilde{C}\subset V$, for some $n \geq 1$, $x^{*}_i \in E^*$, and $\lambda_i \in \R$. Now define $W:=\cap_{i=1}^n\{x^{*}_i<\lambda_i\}$, which verifies $W\cap C=W^{**}\cap C\subset V\cap C\subset U\cap C$.

$\Rightarrow$ Let $V \in \beta(E^{**},E^*)(\ceroe)$, it is not restrictive to assume that $V=A^{\circ}=\{x^{**}\in E^{**}\colon |x^{**}(x^*)|\leq 1,\, \forall x^* \in A\}$ for some $\beta(E^*,E)$-bounded subset $A\subset E^*$. Since $E$ is infrabarreled, $U:= J^{-1}_{E}(V)=A_{\circ} \in \tau(\ceroe)$, i.e., $U=\{x \in E \colon |x^{*}(x)|\leq 1,\, \forall x^* \in A\}$. Now, as $\ceroe \in  \mbox{PC}(C)$, there exists $W_1=\cap_{i=1}^n\{x^{*}_i \leq \alpha_i\}\in \sigma(E,E^*)(\ceroe)$ such that $W_1\cap C\subset U \cap C$, for some $n \geq 1$, $x^{*}_i \in E^*$, and $\alpha_i \in \R$. Define   $W_2:=\cap_{i=1}^n\{x^{*}_i \leq \alpha_i\}^{**}\in \sigma(E^{**},E^*)(\ceroe)$, we will check that $\widetilde{C}\cap W_2\subset \widetilde{C}\cap V$. Fix an arbitrary $x^{**} \in \widetilde{C}\cap W_2$. Then,  there exists a net $(c_{\gamma})_{\gamma \in \Gamma} \subset C$ converging to $x^{**}$ under the topology $\sigma(E^{**},E^*)$. It is not restrictive to assume that $x^{*}_i(c_{\gamma}) \leq \alpha_i$ for every $i$ and $\gamma$. Then, for every $\gamma$, we have $c_{\gamma}\in W_1\cap C \subset U\cap C$, which leads to $x^{**} \in V$. Indeed, if $x^*\in A$, then $|x^{**}(x^*)|=\lim_{\gamma}|x^{*}(c_{\gamma})|\leq 1$, because $(c_{\gamma})_{\gamma}\subset U=A_{\circ}$. As a consequence, $x^{**}\in A^{\circ}=V$ and the proof is over.
\end{proof}
\begin{pro}\label{Prop_C_tilde_pointed_C_bidual_pointed_convex}
Let $E$ be an infrabarreled  l.c.s. and $C \subset E$ a convex subset. Then $PC(C)\cap \mbox{ext}(\widehat{C})\subset \mbox{ext}(\widetilde{C})$.
\end{pro}
\begin{proof}
Fix $x \in \mbox{PC}(C)\cap \mbox{ext}(\widehat{C})$. We first assume $x=\ceroe$. If $\ceroe \not \in \mbox{ext}(\widetilde{C})$, then there exist $x^{**}$, $y^{**} \in \widetilde{C}$ such that $\ceroe =\frac{x^{**}+y^{**}}{2}$.  We fix two nets in $C$, $(x_{\gamma})_{\gamma \in \Gamma}$ and  $(y_{\gamma})_{\gamma \in \Gamma}$, converging respectively to $x^{**}$ and $y^{**}$ under the topology $\sigma(E^{**},E^*)$. Then $\ceroe=\lim_{\gamma}\frac{x_{\gamma}+y^{**}}{2}$ under the topology $\sigma(E^{**},E^*)$. By Proposition \ref{Prop_PC_equivale_weak_star_PC_convex} we have $\ceroe\in\sigma(E^{**},E^*)\mbox{-PC}(\widetilde{C})$ and, as a consequence, $\ceroe=\lim_{\gamma}\frac{x_{\gamma}+y^{**}}{2}$ under the topology $\beta(E^{**},E^*)$. Then $x^{**}=-y^{**}=\lim_{\gamma}x_{\gamma}$ under the topology $\beta(E^{**},E^*)$, which yields $x^{**} \in \widehat{C}$. Analogously we have $y^{**} \in \widehat{C}$. Now, since $\ceroe =\frac{x^{**}+y^{**}}{2}$, it follows that $\ceroe \not \in \mbox{ext}(\widehat{C})$. Now, assume that $x\not =\ceroe$. It follows that $\ceroe \in \mbox{PC}(C-x)\cap \mbox{ext}(\widehat{C}-x)$. By the former paragraph we have $\ceroe \in \mbox{ext}(\widetilde{C}-x)$. Thus $x \in \mbox{ext}(\widetilde{C})$.
\end{proof}
Next, a new characterization of denting points which transforms the inclusion in Theorem~\ref{Teo_dentability_ELC_C_convex} into an equality. This is a consequence of Theorems~\ref{Teo_Generalizacion_dentabilidad_convexos_RACSAM} and \ref{Teo_dentability_ELC_C_convex},  and Proposition~\ref{Prop_C_tilde_pointed_C_bidual_pointed_convex}.
\begin{theo}\label{Teo_dentabilidad_convexos_nuevo}
Let $E$ be an infrabarreled l.c.s. and $C \subset E$ a convex subset. Then $\mbox{b}DP(C)=bPC(C)\cap \mbox{ext}(\widehat{C})$.
\end{theo}
As a consequence, the inclusion in Theorem~\ref{Teo_dentability_ELC_C_convex} becomes an equality when dealing with bounded sets.
\begin{coro}\label{Coro_dentabilidad_convexos_acotados}
Let $E$ be an infrabarreled l.c.s. and $C \subset E$ a bounded convex subset. Then $DP(C)=PC(C)\cap \mbox{ext}(\widehat{C})$.
\end{coro}
Next, we state the version for normed spaces of Theorem \ref{Teo_dentabilidad_convexos_nuevo} which also provides a new characterization in such a context. Let us note that dealing with normed spaces, the map $J_E$ becomes the canonical inclusion from the normed space into its bidual Banach space. Let us fix some notation and terminology in such a context. From now,  $X$ will stand for a normed space and $\norm$ for its norm. In this framework, instead of $J_E$ we consider the canonical inclusion $i_X:(X,\norm)\rightarrow (X^{**},\norm^{**})$ which is an isometric isomorphism. Let us denote $\widehat{X}:=\overline{i_X(X)}^{(X^{**},\norm^{**})}$, we recall that the Banach space $(\widehat{X},\norm^{**})$ is a called the completion of $(X,\norm)$. Analogously, we define the completion of the set $A\subset X$ by $\widehat{A}:=\overline{i_X(A)}^{(X^{**},\norm^{**})}\subset \widehat{X}\subset X^{**}$. It follows easily that a normed space $X$ is a Banach space if and only if $X=\widehat{X}$; in such a case $\widehat{A}=\overline{A}$, i.e., the closure of $A$ on $(X,\norm)$.

\begin{coro}\label{Tma_dentability_convex_nuevo}
Let $X$ be a normed space and $C \subset X$ a convex subset. Then $DP(C)=PC(C)\cap \mbox{ext}(\widehat{C})$.
\end{coro}

The version for Banach spaces of Theorem \ref{Teo_dentabilidad_convexos_nuevo} contains, as a particular case, the well-known characterization of denting points by Lin-Lin-Troyanski in \cite{Lin1988}.

\begin{coro}\label{Tma_dentability_convex_contiene_LinLinTro}
Let $Y$ be a Banach space and $C \subset Y$ a convex subset. Then $DP(C)=PC(C)\cap \mbox{ext}(\overline{C})$.
\end{coro}

\section{Weak property ($\pi$) of cones in locally convex spaces}\label{Section_cones_LCS}
In the first part of this section we state the version for cones of the results from the former section. Later, we obtain some characterizations regarding dentability of cones which are not consequences of the results of the former section. Among other results, we obtain some equivalences for the weak property (\(\pi\)), for the angle property, and some other related.

Let us introduce some terminology. A cone $\cK$ in $E$ is a convex subset such that $\lambda k \in \cK$ for every $\lambda \geq 0$ and $k \in \cK$. We say that $\cK$ is pointed if $\cK\cap (-\cK)=\{\ceroe\}$ or equivalently if $\ceroe \in \mbox{ext}(\cK)$. 
Given a cone $\mathcal{K}$ in a l.c.s. $E$, we define the dual cone of $\mathcal{K}$ by
\(\mathcal{K}^*:=\{x^{*} \in E^*\colon x^{*}(k) \geq 0,\, \forall k \in \mathcal{K}\}\subset E^*\),
and the bidual cone of $\mathcal{K}$ by
\(\mathcal{K}^{**}:=\{x^{**} \in E^{**} \colon x^{**}(x^{*}) \geq 0,\, \forall x^{*} \in \mathcal{K}^*\}\subset E^{**}\). Given a set $A \subset E$, we denote the linear hull of $A$ by span$(A)$.
\begin{pro}\label{Prop_cierre_w*_cono_es_cono_bidual}
Let $E$ be a l.c.s. and $\cK\subset E$ a cone. The following statements hold true.
\item[(i)] $\widetilde{\mathcal{K}}=\mathcal{K}^{**}$.
\item[(ii)] $\mathcal{K}^{**}$ is pointed if and only if  span$(\mathcal{K}^{*})$ is $\beta(E^*,E)$-dense in $E^*$
\end{pro}
\begin{proof}
(i) It is sufficient to show that $\cK^{**} \subset \widetilde{\mathcal{K}}$. Assume that there exists some $x^{**} \in \cK^{**} \setminus  \widetilde{\mathcal{K}}$. Since $\sigma(E^{**},E^*)$ is compatible with the dual par $<E^*,E^{**}>$, we can apply \cite[Theorem 3.32]{Montensinos2ndbook} and choose $x^{*} \in E^*$ and $\alpha \in \R$ such that $x^{**}(x^{*})<\alpha \leq \inf\{y^{**}(x^{*})\colon y^{**} \in \widetilde{\mathcal{K}}\}\leq \inf\{x^{*}(k)\colon k \in \cK\} \leq 0$. Since $\cK$ is a cone, for every $\lambda>0$ and $k \in \cK$, $\lambda k \in \cK$. Hence $\inf\{x^*(k)\colon k \in \cK\}\geq \frac{\alpha}{\lambda}$, for every $\lambda>0$. Now, taking limit as $\lambda \rightarrow +\infty$, we have that $\inf\{x^*(k)\colon k \in \cK\}=0$. This implies that $x^*\in \cK^*$, $\alpha \leq 0$, and that $x^{**}(x^*)<0$. The latter contradicts $x^{**}\in \cK^{**}$.
%
%
%
%

(ii)  $\Rightarrow$ Assume that there exists some $x^{*} \in E^*\setminus \overline{\mbox{span}}^{\beta(E^*,E)}(\cK^*)$. Then, there exist $x^{**} \in E^{**}$, $x^{**} \not = 0$, and $\alpha \in \R$ such that $x^{**}(x^{*})<\alpha \leq \inf\{x^{**}(y^{*})\colon y^{*} \in \overline{\mbox{span}}^{\beta(E^*,E)}(\cK^*)\}$. Thus $x^{**}(y^{*})=0$, $\forall y^{*} \in \overline{\mbox{span}}^{\beta(E^*,E)}(\cK^*)$. As a consequence, $x^{**} \in \cK^{**}\cap (-\cK^{**})$. As $\cK^{**}$ is pointed we have $x^{**}=0$,  a contradiction.

$\Leftarrow$ Let $x^{**} \in \cK^{**}\cap (-\cK^{**})$. Then $0 \leq x^{**}(x^{*}) \leq 0$, $\forall x^{*} \in \cK^*$. Thus $x^{**}=0$ on span$(\cK^*)$. Now, continuity of $x^{**}$ on $(E^*,\beta(E^*,E))$ yields $x^{**}=0$ on $E^*$. It follows that $\cK^{**}$ is pointed.
\end{proof}

Next, we adapt from general convex sets to cones some results from the previous section. These adaptations, together with Proposition \ref{Prop_cierre_w*_cono_es_cono_bidual}, provide new characterizations for the weak property ($\pi$). Let us note that the condition $\ceroe \in \mbox{PC}(\cK)$ in the definition of bounded point of continuity (Definition~\ref{Defi_PointContinuity_y_boundedPC} (ii)) is redundant when it is applied to cones. Indeed, fix $\cK$ a cone and $U \in \tau(0)$ (assumed balanced). Let $r > 0$ be such that $W\cap \cK\subset rU$. It is not restrictive to assume that $W=\cap_{i=1}^n\{x^{*}_i\leq \alpha_i\}$, for some $x^{*}_i \in E^*$, $\alpha_i \in \R$, $i\in \{1,\ldots,n\}$, $n \in \N$. Then $\cap_{i=1}^n\{x^{*}_i \leq \frac{\alpha_i}{r}\}\cap \cK \subset U$.  By Theorem~\ref{Teo_dentability_ELC_widetildeC} and  Proposition \ref{Prop_cierre_w*_cono_es_cono_bidual} (i) we have. 
\begin{theo}\label{Teo_dentability_ELC_K_bidual}
Let $E$ be a l.c.s. and $\mathcal{K} \subset E$ a cone. If $\ceroe \in DP(\mathcal{K})$, then $\ceroe\in PC(\mathcal{K})$ and $\cK^{**}$ is pointed.
\end{theo}
Before establishing the reverse of the former result, we state  some related results. First, we state the version for cones of Proposition \ref{Prop_pseudoslices_base_LCS_convex_set}. 

\begin{pro}\label{Prop_pseudoslices_base_LCS}
Let $E$ be a l.c.s. and $\mathcal{K} \subset E$ a cone. If  $\ceroe \in bPC(\mathcal{K})$ and $\cK^{**}$ is pointed, then there exists  $U \in \tau(\ceroe)$ such that for every $0 <r\leq 1$: 
\item[(i)] $\cK\cap rU$ is bounded.
\item[(ii)] The family of open slices of $\cK\cap rU$ containing $\ceroe$ forms a neighbourhood base of $\ceroe$ for the topology $\sigma(E,E^*)$ relative to $\cK\cap rU$.
\end{pro}

The next result is the version for cones of Theorem \ref{Teo_caract_extremal_C_tilde}. This generalizes \cite[Proposition 2.5]{GARCIACASTANO20151178} to the context of locally convex spaces solving \cite[Problem~2.6]{GARCIACASTANO20151178} in this new framework. 

\begin{theo} 
Let $E$ be a l.c.s. and $\mathcal{K} \subset E$ a cone. Assume that there exists $U \in \tau(\ceroe)$ such that $\cK\cap U$ is bounded. The following are equivalent.
\item[(i)] $\ceroe \in \sigma(E,E^*)\mbox{-s-ext}(\cK\cap U)$.
\item[(ii)] $\cK^{**}$ is pointed.
\item[(iii)] The family of open slices of $\cK\cap U$ containing $\ceroe$ forms a neighbourhood base of $\ceroe$ for the topology $\sigma(E,E^*)$ relative to $\cK\cap U$.
\end{theo}

\begin{defin}
Let $E$ be a l.c.s. and $\mathcal{K} \subset E$ a cone. We say that $\cK$ satisfies the weak property ($\pi$) if there exists $x^{*} \in \cK^*$ and $\alpha >0$ such that $\cK\cap \{x^{*}\leq \alpha\}$ is bounded.
\end{defin}

It is straightforward to show that $\cK$ satisfies the weak property ($\pi$) if and only if $\ceroe \in \mbox{bDP}(\cK)$. However, we will use the terminology  weak property ($\pi$) just introduced because it is the usual one in the context of geometric properties of cones. It is also clear that if $\cK$ satisfies the weak property ($\pi$), then  $\cK$ is a pointed cone and  $\ceroe \in \mbox{DP}(\mathcal{K})$. Next, the version for cones of Theorem \ref{Teo_dentability_ELC_cond_fuertePC_convex}.  In such a result we strengthen the assumptions in Theorem~\ref{Teo_dentability_ELC_K_bidual} in order to provide a corresponding reverse. 

\begin{theo}\label{Teo_dentability_ELC_cond_fuertePC}
Let $E$ be a l.c.s. and $\mathcal{K} \subset E$ a cone. If $\ceroe \in bPC(\mathcal{K})$ and  $\cK^{**}$ is pointed, then $\cK$ satisfies the weak property ($\pi$). 
\end{theo}

In \cite{GARCIACASTANO20151178} and \cite{GARCIACASTANO2018} several characterizations of the notion of dentability of cones in the framework of normed spaces were stated. In the following result we state  the generalizations of some of them in the context of locally convex spaces. Let us note that we do not generalize all of them because, either some have been already obtained in \cite{Qiu2001,SONG2003308}, or they do not remain true as equivalences in locally convex spaces. The latter will be seen in the subsequent result. The next theorem does not need to be proved. Equivalence (i) to (v) are consequence of Theorem \ref{Teo_Generalizacion_dentabilidad_convexos_RACSAM} and Proposition \ref{Prop_cierre_w*_cono_es_cono_bidual}. Equivalence (i)$\Leftrightarrow$(vi) can be easily proved.
\begin{theo}\label{Teo_Generalizacion_JMAA_RACSAM_lcs}
Let $E$ be a l.c.s. and $\mathcal{K} \subset E$ a pointed cone. The following are equivalent.
\item[(i)]  $\cK$ satisfies the weak property ($\pi$).
\item[(ii)] $\ceroe \in bPC(\mathcal{K})\cap DP(\mathcal{K})$.
\item[(iii)] $\ceroe \in bPC(\mathcal{K})\cap \sigma(E,E^*)\mbox{-s-ext}(\cK)$.
\item[(iv)] $\ceroe \in bPC(\mathcal{K})$ and  $\mathcal{K}^{**}$ is pointed.
\item[(v)] $\ceroe \in bPC(\mathcal{K})$ and  span$(\mathcal{K}^{*})$ is $\beta(E^*,E)$-dense in $E^*$.
\item[(vi)] There exist $n \geq 1$, $x^{*}_i \in \cK^*$, $i \in \{1,\ldots,n\}$, and $\lambda >0$, such that $\cap_{i=1}^n\{x^{*}_i < \lambda\}\cap \cK$ is bounded.
\end{theo}

Next, we study some properties which characterize weak property ($\pi$) in normed spaces, see \cite{GARCIACASTANO20151178,GARCIACASTANO2018}, but which do not characterize weak property ($\pi$) in locally convex spaces unless we add an extra assumption. Previously, we introduce some terminology. The order on $E$ given by the cone $\cK$ is defined as $x \leq y \Leftrightarrow y-x \in \cK$ for arbitrary $x$, $y \in E$. In this context it is defined the order interval $[x,y]$ as the set $\{z \in E \colon x \leq z \leq y\}$, and it is said that $k \in \cK$ is an order unit of $\cK$ if $E=\cup_{n \geq 1}[-nk,nk]$. Let us recall that a set in a topological space is said to be rare (or nowhere dense) if its closure has an empty interior. Finally, a l.c.s. $E$ is said to be a Baire-like space if $E$ can not be covered by any increasing sequence of rare, balanced, convex sets. Metrizable barreled vector spaces and Fréchet spaces are examples of Baire-like spaces  (see \cite{Todd1973} for details). 
\begin{theo}\label{Teo_Condi_RACSAM_sufic_elc}
Let $E$ be a l.c.s. and $\mathcal{K} \subset E$ a pointed cone. Consider the following statements.
\begin{itemize}
\item[(i)] $\cK$ satisfies the weak property ($\pi$).
\item[(ii)] $\cK^*$ has an order unit.
\item[(iii)] There exists a sequence $\{x^{*}_n\}_n\subset \cK^*$ such that $E^*=\cup_{n \geq 1}[-nx^{*}_n,nx^{*}_n]$.
\end{itemize}
Then (i)$\Rightarrow$(ii)$\Rightarrow$(iii). Furthermore, if $(E^*,\beta(E^*,E))$ is a Baire-like space, then (iii)$\Rightarrow$(i).
\end{theo}
\begin{proof}
(i)$\Rightarrow$(ii) By \cite[Theorem 2.1]{Qiu2001} there exits some $x^{*}$ in the interior  of $\cK^*$ under the $\beta(E^*,E)$ topology. Thus, there exists some balanced $V \in \beta(E^*,E)(\ceroe)$ such that $x^{*}+V \subset \cK^*$. We will show that $x^{*}$ is an order unit. Indeed, given any arbitrary $y^{*} \in E^*$, there exists $\lambda>0$ such that $y^{*} \in \lambda V=\lambda (-V)\subset \lambda x^{*}-\lambda \cK^*$. This gives $y^{*} \leq \lambda x^{*}$.

(ii)$\Rightarrow$(iii) There is nothing to prove.

(iii)$\Rightarrow$(i) It is not restrictive to assume that the union $\cup_{n \geq 1}[-nx^{*}_n,nx^{*}_n]$ is increasing, otherwise we can define $y^{*}_n :=\sum_{i=1}^nx^{*}_i$ and clearly we have $E^*=\cup_{n \geq 1}[-ny^{*}_n,ny^{*}_n]$. Since each interval $[-nx^{*}_n,nx^{*}_n]$ is convex and balanced, the assumption $E^*$ is a Baire-like space implies that there exist $n_0$ and a balanced $V \in \beta(E^*,E)(\ceroe)$ such that $V \subset [-n_0x^{*}_{n_0},n_0x^{*}_{n_0}]=(-n_0x^{*}_{n_0}+\cK^*)\cap (n_0x^{*}_{n_0}-\cK^*)$. Hence $x^{*}_{n_0}+\frac{1}{n_0} V \subset \cK^*$, which yields that $x^{*}_{n_0}$ belongs to the interior of $\cK^*$, and again \cite[Theorem 2.1]{Qiu2001} applies.
\end{proof}
If $E$ is a normed space, the topology $\beta(E^*,E)$ on $E^*$ coincides with the topology generated by the dual norm. In addition, the dual of a normed space is a Banach space, and so it is a Baire-like space.  Therefore, the former result applied to normed spaces provides us the equivalence (i)$\Leftrightarrow$(iv)$\Leftrightarrow$(v) in \cite[Theorem 2]{GARCIACASTANO2018}.

Normality of a cone may be characterized in terms of the behaviour of its dual. In particular, \cite[Theorem 2.26]{ali-tour} states that a cone in a l.c.s. is normal if and only if the span of the dual cone is the whole dual space. Next, we provide a new characterization of normality which relaxes the former assumption on the dual cone. This is consequence of Theorem \ref{Teo_Generalizacion_JMAA_RACSAM_lcs} and \cite[Theorem 1 (iv)]{SONG2003308}.
\begin{coro}\label{Coro_caract_conos_normales}
Let $E$ be a l.c.s. and $\mathcal{K} \subset E$ a cone such that $\ceroe \in bPC(\mathcal{K})$. Then $\cK$ is normal if and only if span$(\mathcal{K}^{*})$ is $\beta(E^*,E)$-dense in $E^*$.
\end{coro}

Let us recall that $\widehat{\mathcal{K}}$ denotes the closure of $J_{E}(\cK)$ (the image of $\cK$ under the map $J_{E}:E \rightarrow E^{**}$) in $E^{**}$ under the topology $\beta(E^{**},E^*)$. Next, we state the version for cones of Theorem \ref{Teo_dentability_ELC_C_convex}.

\begin{theo}\label{Teo_dentability_ELC_K_tilde}
Let $E$ be a l.c.s. and $\mathcal{K} \subset E$ a cone. If $\ceroe \in DP(\mathcal{K})$, then $\ceroe\in PC(\mathcal{K})$ and $\widehat{\mathcal{K}}$ is pointed.
\end{theo}

The following theorem is the version for cones of Theorem \ref{Teo_dentabilidad_convexos_nuevo}. This provides a new characterization of weak property ($\pi$) for infrabarreled spaces. This characterization is not a generalization of any known characterization of dentability in normed spaces. 
\begin{theo}\label{Teo_weak_property_completitud_pointed_lcs}
Let $E$ be an infrabarreled l.c.s. and $\mathcal{K} \subset E$ a pointed cone. The following are equivalent.
\item[(i)]  $\cK$ satisfies the weak property ($\pi$).
\item[(ii)] $\ceroe \in bPC(\mathcal{K})$ and  $\widehat{\mathcal{K}}$ is pointed.
\end{theo}
%

In what follows, we will state some other new results characterizing notions linked to weak property ($\pi$). Let us note that in Definition \ref{Defi_PointContinuity_y_boundedPC} (ii), we introduced a new definition of point of continuity stronger than the original notion of point of continuity (in Definition \ref{Defi_PointContinuity_y_boundedPC} (i)). Several kinds of strengthening of the notion of point of continuity are also found  in some results in \cite{Qiu2001}, they allow the author to characterize the solidness of the dual cone under different topologies on the dual $E^*$. Next, we add new equivalences to \cite[Theorem 2.2]{Qiu2001}  which characterize the strong angle property. We do that after introducing a new strengthening of the notion of point of continuity. So, we need to establish more terminology. Given a l.c.s. $E$, we will consider the strong topology on it, denoted by $\beta(E,E^*)$, which is defined as the topology on $E$ of uniform convergence on all the $\beta(E^{*},E)$ bounded subsets of $E^*$. A base of neighbourhoods of $\ceroe$ for the topology $\beta(E,E^*)$ is given by the barrels in $E$, i.e., by the family of closed, convex, balanced and absorbent subsets of $E$. It is worth pointing out that the strong topology is the finest one on $E$ which can be defined in terms of the dual pair $<E,E^*>$. We refer the reader to \cite{Kothe1983} for more details. 

\begin{defin}\label{Def_strong_bounded_point_continuity}
Let $E$ be a l.c.s. and $\cK \subset E$ a cone. We say that $\ceroe$ is a $\beta(E, E^*)$ bounded point of continuity of $\cK$, written $\ceroe \in \beta(E, E^*)\mbox{-bPC}(\mathcal{K})$, if there exists $W \in \sigma(E, E^*)(\ceroe)$ such that $\cK \cap W$ is $\beta(E, E^*)$-bounded. 
\end{defin}
Clearly, $\beta(E, E^*)\mbox{-bPC}(\mathcal{K})\subset $ bPC$(\mathcal{K})$. Next, our characterization for the strong angle property, the definition of which corresponds to statement (i).

\begin{theo}\label{Teo_Caract_strong_angle_property}
Let $E$ be an infrabarreled l.c.s. and $\mathcal{K} \subset E$ a pointed cone. The following are equivalent.
\item[(i)] There exists $x^{*} \in \cK^*$ and $\lambda \in \R$ such that $\{x^{*}\leq \lambda\}\cap \cK$ is $\beta(E, E^*)$-bounded.
\item[(ii)] $\ceroe \in \beta(E, E^*)\mbox{-b}PC(\mathcal{K})$ and  $\widehat{\mathcal{K}}$ is pointed.
\item[(iii)] $\ceroe \in \beta(E, E^*)\mbox{-b}PC(\mathcal{K})\cap \sigma(E,E^*)\mbox{-s-ext}(\cK)$.
\item[(iv)] $\ceroe \in \beta(E, E^*)\mbox{-b}PC(\mathcal{K})$ and  $\mathcal{K}^{**}$ is pointed.
\item[(v)] $\ceroe \in \beta(E, E^*)\mbox{-b}PC(\mathcal{K})$ and  span$(\mathcal{K}^{*})$ is $\beta(E^*,E)$-dense in $E^*$.
\item[(vi)] There exist $n \geq 1$, $x^{*}_i \in \cK^*$, $i \in \{1,\ldots,n\}$, and $\lambda >0$, such that $\cap_{i=1}^n\{x^{*}_i < \lambda\}\cap \cK$ is $\beta(E, E^*)$-bounded.
\end{theo}
\begin{proof}
(i)$\Rightarrow$(ii) Clearly $\ceroe \in \beta(E, E^*)\mbox{-bPC}(\mathcal{K})$. On the other hand, since every $\beta(E, E^*)$-bounded set is bounded regarding the initial topology on $E$, it follows that $\cK$ satisfies the weak property ($\pi$). Now by Theorem \ref{Teo_weak_property_completitud_pointed_lcs}, $\widehat{\mathcal{K}}$ is pointed.

(ii)$\Rightarrow$(i) By Propositions \ref{Prop_C_tilde_pointed_C_bidual_pointed_convex} and \ref{Prop_cierre_w*_cono_es_cono_bidual}, $\cK^{**}$ is pointed. On the other hand, we consider $W \in \sigma(E, E^*)(\ceroe)$ such that $\cK\cap W$ is $\beta(E, E^*)$-bounded. Then there exists $U \in \tau(\ceroe)$ closed such that $\cK\cap U \subset \cK \cap W$. Thus $\cK\cap U $ is $\beta(E, E^*)$-bounded (and so bounded). Now we argue as in proof of Proposition~\ref{Prop_pseudoslices_base_LCS_convex_set}~(ii) and claim that the family of open slices containing $\ceroe$ forms a neighbourhood base for $\ceroe$ relative to $(\cK\cap U,\mbox{weak})$. Now, arguing as in the proof of Theorem \ref{Teo_dentability_ELC_cond_fuertePC_convex}, there exists $x^{*} \in \cK^*$ and $\lambda \in \R$ such that $\{x^{*}\leq \lambda\}\cap \cK \subset \cK\cap U$. Thus, $\{x^{*}\leq \lambda\}\cap \cK$ is $\beta(E, E^*)$-bounded.

(ii)$\Leftrightarrow$(iii)$\Leftrightarrow$(iv)$\Leftrightarrow$(v)$\Leftrightarrow$(vi) Consequence of Theorem \ref{Teo_Generalizacion_JMAA_RACSAM_lcs}.
\end{proof}

In the next result we will use that the topology of infrabarreled spaces coincides with the topology of uniform convergence on strongly bounded sets of $E^*$, i.e., $E$ is infrabarreled if and only if the initial topology, $\tau$, on $E$ verifies $\tau=\beta^*(E, E^*)$. This property allows us to state easily the dual version of Theorem \ref{Teo_Caract_strong_angle_property}. Next, we introduce the dual notion of bounded point of continuity.

\begin{defin}\label{Def_dual_strong_bounded_point_continuity}
Let $E$ be a l.c.s., $\cK \subset E$ a cone, and $\tau^*$ a topology on $E^*$. We say that $\ceroe$ is a $\tau^*$ bounded point of continuity of $\cK^*$, written $\ceroe \in \tau^*\mbox{-b}PC(\mathcal{K^*})$, if there exists $W \in \sigma(E^*, E)(\ceroe)$ such that $\cK^* \cap W$ is $\tau^*$-bounded. 
\end{defin}

\begin{theo}\label{Teo_Caract_strong_angle_property_dual}
Let $E$ be an infrabarreled l.c.s. and $\mathcal{K} \subset E$ a closed pointed cone. The following are equivalent.
\item[(i)] There exists $k \in \cK$ and $\lambda \in \R$ such that $\{k\leq \lambda\}\cap \cK^*$ is $\beta(E^*, E)$-bounded.
\item[(ii)] $\ceroe \in \beta(E^*, E)\mbox{-b}PC(\mathcal{K^*})\cap\sigma(E^*, E)\mbox{-s-ext}(\cK^*)$.
\item[(iii)] $\ceroe \in \beta(E^*, E)\mbox{-b}PC(\mathcal{K^*})$ and  span$(\mathcal{K})$ is dense in $E$.
\item[(iv)] There exist $n \geq 1$, $k_i \in \cK$, $i \in \{1,\ldots,n\}$, and $\lambda >0$, such that $\cap_{i=1}^n\{k_i < \lambda\}\cap \cK^*$ is $\beta(E^*, E)$-bounded.
\end{theo}

We finish this section providing a dual version of Theorems \ref{Teo_Generalizacion_JMAA_RACSAM_lcs} and \ref{Teo_Condi_RACSAM_sufic_elc}. Let us recall that the original topology on a barreled l.c.s. $E$ coincides with both the strong and the Mackey topology --resp. $\beta(E,E^*)$ and $\uptau(E,E^*)$-- and, in addition, every bounded subset of $E^*$ is $\sigma(E^*,E)$ relatively compact. On the other hand, a topology $\tau^*$ on $E^*$ is compatible with the dual pair $<E,E^*>$ if and only if $\sigma(E^*,E) \preceq \tau^* \preceq \uptau(E^*,E)$, where $\uptau(E^*,E)$ denotes the Mackey topology on $E^*$,  (see \cite{Kothe1983} for details). Our next result provides more equivalences to the characterization in \cite[Theorem 3.1]{Qiu2001} of whether a dual cone satisfies the weak property ($\pi$) under the topology $\sigma(E^*,E)$. 
\begin{theo}\label{Tma_carac_dual_cone_barreled_elc}
Let $E$ be a barreled l.c.s., $\cK\subset E$ a cone, and $\tau^*$ a topology on $E^*$ compatible with the dual pair $<E,E^*>$. The following are equivalent.
\item[(i)] $\cK^*$ satisfies the weak property $(\pi)$ in $(E^*,\tau^*)$.
\item[(ii)] $\ceroe \in \tau^*\mbox{-b}PC(\mathcal{K^*})\cap \sigma(E^*,E)\mbox{-s-ext}(\cK^*)$.
\item[(iii)]  $\ceroe \in \tau^*\mbox{-b}PC(\mathcal{K^*})$ and  span$(\mathcal{K})$ is dense in $E$.
\item[(iv)]  $\ceroe \in \tau^*\mbox{-b}PC(\mathcal{K^*})$ and  $\cK^*$ is pointed.
\item[(v)] There exist $n \geq 1$, $x_i \in \cK$, $i \in \{1,\ldots,n\}$, and $\lambda >0$, such that $\cap_{i=1}^n\{x_i < \lambda\}\cap \cK^*$ is bounded.
\end{theo}
\begin{proof}
(i)$\Leftrightarrow$(ii)$\Leftrightarrow$(v) Is a consequence of Theorem \ref{Teo_Generalizacion_JMAA_RACSAM_lcs}.

(i)$\Leftrightarrow$(iii) We apply Theorem \ref{Teo_Generalizacion_JMAA_RACSAM_lcs} and the following equalities $\cK^{**}=\{x\in E \colon x^{*}(x)\geq 0,\, \forall x^{*} \in \cK^*\}=\overline{\cK}$, and $\overline{\mbox{span}}(\cK)=\overline{\mbox{span}}(\overline{\cK})$. Indeed, span$(\overline{\cK})\subset\overline{\mbox{span}}(\cK)$. Hence $\overline{\mbox{span}}(\overline{\cK})\subset\overline{\mbox{span}}(\cK)\subset \overline{\mbox{span}}(\overline{\cK})$, which gives the last equality.

(iii)$\Leftrightarrow$(iv) We apply that span$(\mathcal{K})$ is dense in $E$ if and only if  $\cK^*$ is pointed.
\end{proof}
As a consequence of Theorem \ref{Teo_Condi_RACSAM_sufic_elc}, we have the following.
\begin{theo}\label{Tma_neces_dual_cone_barreled_elc}
Let $E$ be a barreled l.c.s., $\cK\subset E$ a cone, and $\tau^*$ a topology on $E^*$ compatible with the dual pair $<E,E^*>$. Consider the following statements.
\begin{itemize}
\item[(i)] $\cK^*$ satisfies the weak property $(\pi)$ in $(E^*,\tau^*)$.
\item[(ii)] $\cK$ has an order unit.
\item[(iii)] There exists a sequence $\{k_n\}_n\subset \cK$ such that $E=\cup_{n \geq 1}[-nk_n,nk_n]$.
\end{itemize}
Then (i)$\Rightarrow$(ii)$\Rightarrow$(iii). Furthermore, if $E$ is a Baire-like space, then (iii)$\Rightarrow$(i).
\end{theo}

\section{Dentability of cones in normed spaces and some related problems}\label{Section_cones_normed_spaces}
In this section we first provide the version for normed spaces of  Theorem~\ref{Teo_weak_property_completitud_pointed_lcs}. It is worth pointing out that in a normed space $X$, a cone $\cK\subset X$ satisfies the weak property ($\pi$) if and only if $\cero \in$ DP$(\cK)$. Let us recall that $\widehat{\mathcal{K}}$ denotes the completion of $\cK$ and $\widehat{X}$ the completion of $X$. So, the next result does not need any proof.
\begin{theo}\label{Teo_denting_equivalente_completitud_cone}
Let $X$ be a normed space and $\mathcal{K}$ a pointed cone. The following are equivalent.
\item[(i)] $\cero \in \mbox{DP}(\mathcal{K})$.
\item[(ii)] $\cero\in \mbox{PC}(\mathcal{K})$ and $\widehat{\mathcal{K}}$ is pointed.
\end{theo}
Theorem \ref{Teo_denting_equivalente_completitud_cone} connects the main results on dentability of cones in \cite{GARCIACASTANO20151178,GARCIACASTANO2018} simplifying significantly their proofs. Indeed, by Proposition \ref{Prop_C_tilde_pointed_C_bidual_pointed_convex} we have $\mbox{PC}(K)\cap \mbox{ext}(\widehat{\cK})\subset \mbox{ext}(\widetilde{\cK})$. On the other hand, $\cero$ is an extreme point of a cone if and only if such a cone is pointed. As a consequence, assertion (ii) in Theorem \ref{Teo_denting_equivalente_completitud_cone} is equivalent to the conditions $\cero\in \mbox{PC}(\mathcal{K})$ and $\widetilde{\cK}$ pointed. Then, Theorem \ref{Teo_denting_equivalente_completitud_cone} leads to (i)$\Leftrightarrow$(iii) of Theorem 1.1 in \cite{GARCIACASTANO20151178}. On the other hand, Proposition~\ref{Prop_cierre_w*_cono_es_cono_bidual} gives $\widetilde{\cK}=\cK^{**}$ and the equivalence between $\cK^{**}$ pointed and span$(\cK^*)$  dense in $X^*$. Then, Theorem \ref{Teo_denting_equivalente_completitud_cone} results in (i)$\Leftrightarrow$(iii) of Theorem~2 in \cite{GARCIACASTANO2018}.

Next, we state an interesting example which will be used in the following. From now on, $c_{00}$ denotes the vector space of real sequences of finite support, i.e., sequences $(x_n)_n\in \R^{\N}$ such that the set $\{n \in \N\colon x_n\not =0\}$ is finite. We consider the following norms on $c_{00}$, $\parallel (x_n)_n \parallel_{\infty}:=\mbox{max}\{|x_n|\colon n \geq 1\}$ and $\parallel (x_n)_n\parallel :=\mbox{max}\{\parallel (x_n)_n \parallel_{\infty},|\sum_{n\geq 3}x_n|\}$. 

\begin{exam}(\cite[Page 315]{SONG2003308})\label{Example_Song}
In the non-complete normed space $(c_{00},\parallel \cdot \parallel)$, the closed and pointed cone $\mathcal{K}:=\{(x_n)_n \in c_{00} \colon x_1 \geq |x_n|, \, \forall n \geq 3, \mbox{ and } x_2=\sum_{n \geq 3}x_n\}$, verifies that $0 \in $ PC$(\mathcal{K})$ and $0 \not \in $ DP$(\mathcal{K})$. 
\end{exam}

\begin{rem}\label{Remark_Clase_cones_mas_grande_PC_equiv_DP}
In normed spaces, the class of cones having a pointed completion can be considered as the largest class of cones for which  $\cero\in$ PC$(\mathcal{K})$ if and only if $\cero \in$ DP$(\mathcal{K})$.
\end{rem}

\begin{proof}[Proof of Remark \ref{Remark_Clase_cones_mas_grande_PC_equiv_DP}]
Let us see that the completion of the cone $\cK$  in Example \ref{Example_Song} is not pointed. Consider the map $T:(c_{00},\norm)\longrightarrow (c_0,\norm_{\infty})$ defined by $T((x_n)):=(\sum_{n>2}x_n,x_1,x_2,x_3,\ldots)$. It is easy to check that $T$ is an isometry and Im~$T=$span$\{e_2,e_3,e_1+e_3,e_1+e_4,e_1+e_5,\ldots\} \subset c_0$.  Now we consider the extension of $T$, denoted by $\widehat{T}$, defined from the completion of $(c_{00},\norm)$, denoted by $(E,\norm_E)$,  on $c_0$, i.e., $\widehat{T}:(E,\norm_E)\longrightarrow (c_0,\norm_{\infty})$. Indeed, $c_0=\overline{\mbox{Im }T}^{\norm_{\infty}}$ because $e_1=\lim_n \frac{(e_1+e_4)+(e_1+e_5)+\cdots+(e_1+e_{3+n})}{n}$. Besides, $\widehat{T}$ is again an isometry. We consider the following two sequences $(y^n)_{n>2}$, $(z^n)_{n>2} \subset \mathcal{K}$ defined by \[y^n_i:=\begin{cases} \frac{1}{n},& i=1,\\ 1,& i=2,\\ \frac{1}{n},& 2<i\leq 2+n,\\0,&2+n<i, \end{cases}\quad z^n_i:=\begin{cases} \phantom{-}\frac{1}{n},& i=1,\\ -1,& i=2,\\ -\frac{1}{n},& 2<i\leq 2+n,\\\phantom{-}0,&2+n<i, \end{cases}\]
where $y^n=(y^n_i)_{i \geq 1}\in \mathcal{K}$  and $z^n=(z^n_i)_{i \geq 1}\in \mathcal{K}$. Both are Cauchy sequences in $(c_{00},\norm)$, then there exists $y$, $z \in \widehat{\mathcal{K}}$ such that $\lim_ny^n =y$ and $\lim_nz^n =z$. It is clear that $y \not=0\not =z$. On the other hand, $\widehat{T}(y)=\lim_nT(y^n)=e_1+e_3$ and $\widehat{T}(z)=\lim_nT(z^n)=-e_1-e_3$. Therefore $z=\widehat{T}^{-1}(-e_1-e_3)=-y \in -\widehat{\mathcal{K}}$. As a consequence $0\not =z \in \widehat{\mathcal{K}} \cap (-\widehat{\mathcal{K}})$. 
\end{proof}
Regarding Banach spaces A. Daniilidis stated in \cite{Daniilidis2000} that in any Banach space $Y$ and for any closed and pointed cone $\mathcal{K}\subset Y$ we have $0_Y \in \mbox{PC}(\mathcal{K})$ if and only if $0_Y \in \mbox{DP}(\mathcal{K})$. Moreover, Example 1.5 in \cite{GARCIACASTANO20151178} shows that the assumption $\cK$ is closed can not be dropped down. Next, we state the version for Banach spaces of Theorem \ref{Teo_weak_property_completitud_pointed_lcs} which contains, as a particular case, the above mentioned chactacterization of Daniilidis in \cite{Daniilidis2000} but not requiring the condition $\cK$ is closed. It is clear that $X$ is a Banach space if and only if $X=\widehat{X}$; in such a case we have  $\widehat{\mathcal{K}}=\overline{\mathcal{K}}$, i.e., $\widehat{\mathcal{K}}$ becomes the closure of $\mathcal{K}$ under the topology generated by the norm of the Banach space.
\begin{theo}\label{Teo_denting_Banach}
Let $\mathcal{K}$ a pointed cone in a Banach space $Y$. The following are equivalent.
\item[(i)] $0_Y \in \mbox{DP}(\mathcal{K})$.
\item[(ii)] $0_Y\in \mbox{PC}(\mathcal{K})$ and $\overline{\mathcal{K}}$ is pointed.
\end{theo}

Next, we discuss some related problems in the literature.  Let us recall that $x \in  \cK$ is said to be a quasi-interior point of $X$  if $\overline{\cup_{n \in \N}[-nx,nx]}=X$, where $[-nx,nx]$ is the order interval given by the order cone, i.e., the set $\{y \in X \colon -nx \leq y \leq nx\}$. The set of all quasi-interior points is denoted by qi$\cK$. If $\cK$ has non empty interior, then the concepts of interior point of $\cK$ and quasi-interior point of $X$ coincide (see \cite{Peressini1967} for details). In \cite{GARCIACASTANO20151178} the authors provided Example 1.5 answering the following problem in the negative.
\begin{prob}(Kountzakis-Polyrakis, \cite[Problem 5]{Kountzakis2006})\label{Prob_Kountzakis}
Let $X$ be a normed space and $\mathcal{K}\subset X$ a pointed cone. Does $\cero\in \mbox{PC}(\mathcal{K})$ imply that qi$\cK^*\not = \emptyset$? 
\end{prob}
In normed spaces, the property $\cero \in$ DP$(\cK)$ is directly connected to the notion of base. 
\begin{defin}
Let $X$ be a normed space and $\mathcal{K} \subset X$ a cone. We say that a convex subset $B \subset \mathcal{K}$ is a base for $\mathcal{K}$ if $\cero \not \in \overline{B}$ and each $k \in \mathcal{K}\setminus\{\cero\}$ has a unique representation of the form $k=\lambda b$ for some $\lambda>0$ and $b \in B$. We say that  $B \subset \mathcal{K}$ is a bounded base for $\cK$ if $B$ is a base for $\mathcal{K}$ and a bounded set on $X$.
\end{defin}
Let us recall that given a cone $\mathcal{K}$, the set of strictly positive functionals on $\mathcal{K}$ is defined by $\mathcal{K}^{\#}:=\{x^{*} \in X^* \colon x^{*}(k) >0, \forall k \in \mathcal{K}\setminus\{\cero\}\}$. It is known that a cone $\mathcal{K}$ has a base if and only if $\mathcal{K}^{\#}\not = \{\emptyset\}$ and that $\mathcal{K}$ has a bounded base if and only if $\cero \in$ DP$(\cK)$. The cone $\cK$ in Example \ref{Example_Song} verifies that $\mathcal{K}^{\#}\not = \emptyset$. As a consequence, by \cite[Theorem 4]{Kountzakis2006}, qi$\mathcal{K}^* = \emptyset$. Therefore, we can state the following.

\begin{rem}\label{Remark_Respuesta_Problemas_Literatura}
The following problems can be answered using Example \ref{Example_Song}. The first one in the positive and the others in the negative.
\end{rem}

\begin{prob}(Gong, \cite[page 629]{Gong1995})\label{Prob_1_Gong}
Let $X$ be a normed space and $\mathcal{K}\subset X$ a closed and pointed cone. Is the condition $0\in \mbox{PC}(\mathcal{K})$ really weaker than the condition that $\cK$ has a bounded base? 
\end{prob}

\begin{prob}(Gong, \cite[page 624]{Gong1995})\label{Prob_2_Gong}
Let $X$ be a normed space and $\mathcal{K}\subset X$ a closed and pointed cone having a base $B$. Does $0\in \mbox{PC}(\mathcal{K})$ imply that $\cK$ has a bounded base?
\end{prob}
 
\begin{prob}(GC-Melguizo Padial, \cite[Problem 1.7]{GARCIACASTANO20151178})\label{Prob_Castano}
Let $X$ be a normed space and $\mathcal{K}\subset X$ a pointed  and closed cone. Does $0\in \mbox{PC}(\mathcal{K})$ imply that qi$\cK^*\not = \emptyset$? 
\end{prob}

Theorem \ref{Teo_denting_equivalente_completitud_cone} and the negative answer to Problem \ref{Prob_2_Gong} leads us to find a particular characterization of the condition $\cero \in \mbox{DP}(\mathcal{K})$ for cones $\mathcal{K}$ having a base. We will devote the rest of this section to such an issue. Let us introduce some more notation. Consider a normed space $X$ and $\mathcal{K} \subset X$ a cone with base $B$. We define the cone generated by $B$ as cone$(B):=\{\lambda b\colon \lambda \geq0,\,b\in B\}$. In addition, we define the generalized recession cone of $B$ by  \({\mbox{R}}(B):=\{{\lim_n}\, t_nb_n \colon t_n \rightarrow 0^+, b_n \in B\}\subset \overline{\mathcal{K}}=\overline{\mbox{cone}}(B)\subset X\). The cone ${\mbox{R}}(B)$ is closed and it is very useful in order to find the closure of $\mbox{cone}(B)$. It is well-known that if $\mathcal{K}$ is a cone with base, then for every $x^{*} \in \mathcal{K}^{\#}$ we have that $B:=\{ k \in \mathcal{K} \colon x^{*}(k)=1\}$ is a base for $\mathcal{K}$. In particular, $\mathcal{K}=\mbox{cone}(B)$. Now consider the completion $(\widehat{X},\widehat{\norm})$ of the normed space $(X,\norm)$, and $\mathcal{K} \subset X$ a cone with base $B$. We define the closed cone \(\mbox{R}_{\widehat{\norm}}(B):=\{\widehat{\lim_n}\, t_nb_n \colon t_n \rightarrow 0^+, b_n \in B\}\subset \widehat{\mbox{cone}}(B)\subset \widehat{X}\),
where by $\widehat{\lim_n}$ we denote the limit on $\widehat{X}$ under the norm ${\widehat{\norm}}$, and by $\widehat{\mbox{cone}}(B)$ the closure on $(\widehat{X},\widehat{\norm})$ of $\mbox{cone}(B)$. Next, the version of Theorem~\ref{Teo_denting_equivalente_completitud_cone} for cones in normed spaces which have a base.
\begin{theo}\label{Teo_Dentability_cone_with_base}
Let $X$ be a normed space and $\mathcal{K} \subset X$ a cone with base $B$. The following are equivalent.
\begin{itemize}
\item[(i)] $\cero \in \mbox{DP}(\mathcal{K})$.
\item[(ii)] $\cero \in \mbox{PC}(\mathcal{K})$ and $\mbox{R}_{\widehat{\norm}}(B)$ is pointed.
\end{itemize}
\end{theo}
Let us note that the condition R$(B)$ pointed does not imply $\mbox{R}_{\widehat{\norm}}(B)$ pointed. For example, the cone $\cK$ in Example~\ref{Example_Song} has $B=\{(x_n)_n \in \cK \colon x_1=1\}$ as a base, ${\mbox{R}}(B)=\{0_X\}$ and $\mbox{R}_{\widehat{\norm}}(B)=\{k\widehat{T}^{-1}(e_1+e_3)\colon k \in \R\}$ is not pointed. On the other hand, even $\mbox{R}(B)$ may not be pointed. For example, the cone $\mathcal{K}=\{(x,y)\in \R^2\colon x>0\}\cup\{(0,0)\}\subset \R^2$ has $B=\{(1,y)\colon y \in \R\}$ as a base and ${\mbox{R}}(B)=\mbox{R}_{\widehat{\norm}}(B)=\{(0,y)\colon y \in \R\}$ is not pointed. 
In the proof of Theorem~\ref{Teo_Dentability_cone_with_base} we will use the following proposition.
\begin{pro}\label{Prop_clausura_cono_recesion}
Let $X$ be a normed space, $\widehat{X}$ its completion, and $\mathcal{K} \subset X$ a cone with base $B$. Then $\widehat{\mbox{cone}}(B)=\mbox{cone}(\widehat{B})\cup \mbox{R}_{\widehat{\norm}}(B)$ and $\mbox{cone}(\widehat{B})$ is pointed.
\end{pro}
\begin{proof}
It is sufficient to show that for any cone $\mathcal{K}$ with a base $B$ in an arbitrary normed space $X$, we have $\overline{\mbox{cone}}(B)=\overline{\mbox{cone}}(\overline{B})=\mbox{cone}(\overline{B})\cup \mbox{R}(B)$ and, in addition, $\mbox{cone}(\overline{B})$ is pointed. Let us check first the equality. Since $\overline{B}\subset\overline{\mbox{cone}}(B)$, we have $\mbox{cone}(\overline{B})\subset \overline{\mbox{cone}}(B)$, thus $\overline{\mbox{cone}}(\overline{B}) \subset \overline{\mbox{cone}}(B)$, which provides \(\overline{\mbox{cone}}(B)=\overline{\mbox{cone}}(\overline{B}) \supset \mbox{cone}(\overline{B})\cup \mbox{R}(B)\). Now the contrary inclusion. For that purpose we fix an arbitrary $x \in \overline{\mbox{cone}}(B)$, then there exists $(\lambda_n b_n)_{n} \subset \mbox{cone}(B)$ such that $x =\lim_n \lambda_n b_n$. We claim that the sequence $(\lambda_n)_n \subset [0,+\infty)$ is bounded. On the contrary, if $\lim_n \lambda_n=+\infty$ (or if it happens for some subsequence), then $\lim_n b_n=\lim_n\frac{x}{\lambda_n}=\cero$ which implies $\cero \in \overline{B}$, a contradiction. Then $(\lambda_n)_n \subset [0,+\infty)$ has a convergent subsequence. For simplicity of notation we assume that $\lim_n \lambda_n=\lambda \geq 0$. If $\lambda=0$, then $x \in \mbox{R}(B)$. In case $\lambda>0$ we have $\lim_n b_n=\lim_n\frac{x}{\lambda_n}=\frac{x}{\lambda} \in \overline{B}$, then $x \in \mbox{cone}(\overline{B})$.

We finish the proof showing that $\mbox{cone}(\overline{B})$ is pointed. Assume the contrary, hence there exists some $x \in \mbox{cone}(\overline{B})\cap (-\mbox{cone}(\overline{B}))$. Then, for some $\lambda \geq 0$, $\mu\geq 0$, $(b_n)_n$, $(\tilde{b}_n)_n \subset B$, we have $x=\lambda \lim_nb_n=-\mu \lim_n\tilde{b}_n$. If $\lambda$, $\mu \not =0$, then $0=\lim_n\left(\frac{\lambda}{\lambda+\mu}b_n+\frac{\mu}{\lambda+\mu}\tilde{b}_n\right) \in \overline{B}$, a contradiction.
\end{proof}

\begin{proof}[Proof of Theorem \ref{Teo_Dentability_cone_with_base}]
After Theorem \ref{Teo_denting_equivalente_completitud_cone}, we only need to prove that the cone, $\widehat{\mbox{cone}}(B)$, is pointed if and only if the cone, $\mbox{R}_{\widehat{\norm}}(B)$, is pointed. The implication ``$\Rightarrow$" is clear because $\mbox{R}_{\widehat{\norm}}(B)\subset \widehat{\mbox{cone}}(B)$. Now, the implication ``$\Leftarrow$". We pick an arbitrary $x\in \widehat{\mbox{cone}}(B) \cap (-\widehat{\mbox{cone}}(B))$ and we will show that $x=\cero$.   If either $x\in \mbox{cone}(\widehat{B})\cap (-\mbox{cone}(\widehat{B}))$ or $x \in \mbox{R}_{\widehat{\norm}}(B) \cap (-\mbox{R}_{\widehat{\norm}}(B))$, then $x=\cero$ because both cones, $\mbox{cone}(\widehat{B})$ and $\mbox{R}_{\widehat{\norm}}(B)$, are pointed. Then, we assume that $x \in \mbox{cone}(\widehat{B})\cap (-\mbox{R}_{\widehat{\norm}}(B))$. Then, there exist $(b_n)_n$, $(b_n^*)_n \subset B$, $\lambda \geq 0$ and $(\lambda_n)_n\subset \R$ such that $\lambda_n>0$ and $\lim_n \lambda_n=0$ verifying $x=\lambda \widehat{\lim_n}b_n=-\widehat{\lim_n} \lambda_nb^*_n$. If $x \not =0$, then $\lambda>0$ and we consider $\bar{b}_n:=\frac{\lambda}{\lambda+\lambda_n}b_n+\frac{\lambda_n}{\lambda+\lambda_n}b^*_n\in B$. Clearly $\lim_n \widehat{\parallel \bar{b}_n\parallel}=\lim_n\frac{\parallel \lambda b_n+\lambda_nb^*_n\parallel}{\lambda+\lambda_n}=0$, a contradiction because $\cero \not \in \widehat{B}$ (since $B$ is a base for $\mathcal{K}$).
\end{proof}

\section*{Acknowledgements}
The authors Fernando Garc\'ia-Casta\~no and M. A. Melguizo Padial have been supported by MINECO and FEDER (MTM2017-86182-P). The authors wish to express their thanks to Antonio Avil\'es and Jos\'e Rodr\'iguez for their help with the proof of Remark \ref{Remark_Clase_cones_mas_grande_PC_equiv_DP}. The authors also gratefully acknowledges the referees for their suggestions that helped to improve the manuscript.


\begin{thebibliography}{10}
\expandafter\ifx\csname url\endcsname\relax
  \def\url#1{\texttt{#1}}\fi
\expandafter\ifx\csname urlprefix\endcsname\relax\def\urlprefix{URL }\fi
\expandafter\ifx\csname href\endcsname\relax
  \def\href#1#2{#2} \def\path#1{#1}\fi

\bibitem{Rieffel:66:Irvine}
M.~Rieffel, {Dentable Subsets of Banach Spaces with Application to a
  Radon-Nikodym Theorem}, in: B.~Gelbaum (Ed.), Proceedings of the Conference
  on Functional Analysis, Thompson Book Co., 1967, pp. 71--77.

\bibitem{Chi1977}
G.~Y.~H. Chi, {On the Radon-Nikod\'ym Theorem and Locally Convex Spaces with
  the Radon-Nikod\'ym Property}, Proceedings of the American Mathematical
  Society 62~(2) (1977) 245--253.

\bibitem{Egghe1978}
L.~Egghe, {On the Radon-Nikod\'ym-Property, and Related Topics in Locally
  Convex Spaces}, in: R.~Aron, S.~Dineen (Eds.), Vector Space Measures and
  Applications II. Lecture Notes in Mathematics, vol 465, Springer Berlin,
  Heidelberg, 1978, pp. 77--90.

\bibitem{Lin1988}
B.-L. Lin, P.-K. Lin, S.~Troyanski, {Characterizations of Denting Points},
  Proceedings of the American Mathematical Society 102~(3) (1988) 526--528.

\bibitem{GARCIACASTANO2018}
F.~Garc{\'{i}}a-Casta\~no, M.~A. Melguizo~Padial, {On Dentability and Cones
  with a Large Dual}, Revista de la Real Academia de Ciencias Exactas,
  F{\'i}sicas y Naturales. Serie A. Matem{\'a}ticas 113~(3) (2019) 2679--2690.

\bibitem{Daniilidis2000}
A.~Daniilidis, Nguyen V.H., Strodiot JJ., Tossings P. (eds). Optimization.
  Lecture Notes in Economics and Mathematical Systems, vol 481, Springer
  Berlin, Heidelberg, 2000, Ch. Arrow-Barankin-Blackwell Theorems and Related
  Results in Cone Duality: A Survey, pp. 119--131.

\bibitem{Gong1995}
X.~H. Gong, {Density of the Set of Positive Proper Minimal Points in the Set of
  Minimal Points}, J. Optim. Theory Appl. 86~(3) (1995) 609--630.

\bibitem{GARCIACASTANO20151178}
F.~Garc{\'{i}}a-Casta\~no, M.~A. Melguizo~Padial, V.~Montesinos, {On Geometry
  of Cones and some Applications}, J. Math. Anal. Appl. 431~(2) (2015) 1178 --
  1189.

\bibitem{Kountzakis2006}
C.~Kountzakis, I.~A. Polyrakis, {Geometry of Cones and an Application in the  Theory of Pareto Efficient Points}, Journal of Mathematical Analysis and
  Applications 320~(1) (2006) 340--351.

\bibitem{Qiu2001}
J.~H. Qiu, {On Solidness of Polar Cones}, Journal of Optimization Theory and
  Applications 109~(1) (2001) 199--214.

\bibitem{SONG2003308}
W.~Song, {Characterizations of some remarkable Classes of Cones}, Journal of
  Mathematical Analysis and Applications 279~(1) (2003) 308 -- 316.

\bibitem{Guirao2014}
A.~J. Guirao, V.~Montesinos, V.~Zizler, On Preserved and Unpreserved Extreme  Points, in: J.~C. Ferrando, M.~L{\'o}pez-Pellicer (Eds.), Descriptive
  Topology and Functional Analysis. Springer Proceedings in Mathematics {\&}
  Statistics, vol 80, Springer, Cham, 2014, pp. 163--193.

\bibitem{Osborne}
M.~Osborne, Locally Convex Spaces, Vol. 269 of Graduate Text in Mathematics,
  Springer International Publishing, 2014.

\bibitem{Montensinos2ndbook}
M.~Fabian, P.~Habala, P.~Hajek, V.~Montesinos~Santalucia, V.~Zizlev, Banach
  Space Theory. The Basis for Linear and Nonlinear Analysis, CMS Books in
  Mathematics, Springer-Verlag New York, 2011.

\bibitem{ali-bor}
C.~Aliprantis, K.~Border, Infinite Dimensional Analysis, Springer-Verlag Berlin
  Heidelberg, 2006.

\bibitem{Todd1973}
A.~R. Todd, S.~A. Saxon, {A Property of Locally Convex Baire Spaces},
  Mathematische Annalen 206~(1) (1973) 23--34.

\bibitem{ali-tour}
C.~Aliprantis, R.~Tourky, Cones and Duality, Graduate Studies in Mathematics,
  American Mathematical Society, 2007.

\bibitem{Kothe1983}
G.~K\"othe, Topological Vector Spaces I, Vol. 159 of Grundlehren der
  mathematischen Wissenschaften, Springer-Verlag Berlin Heidelberg, 1983.

\bibitem{Peressini1967}
A.~Peressini, Ordered Topological Vector Spaces, Harper's Series in Modern
  Mathematics, Harper \& Row, 1967.

\end{thebibliography}
\end{document}